\documentclass[11pt,reqno]{amsart}
\usepackage{amsmath,amsfonts,amsthm,amscd,amssymb,graphicx}
\usepackage{epstopdf}
\usepackage{color}
\numberwithin{equation}{section}
\usepackage{hyperref}
\usepackage{fullpage}

\def\XXint#1#2#3{{\setbox0=\hbox{$#1{#2#3}{\int}$}
\vcenter{\hbox{$#2#3$}}\kern-.5\wd0}}

\newcommand{\R}{\mathbb{R}}
\newcommand{\C}{\mathbb{C}}
\newcommand{\Z}{\mathbb{Z}}
\newcommand{\N}{\mathbb{N}}

\newcommand{\I}{\mathbb{I}}

\newcommand{\qut}[1]{\textquotedblleft #1\textquotedblright}
\newcommand{\be}{\begin{equation}}
\newcommand{\eq}[2]{\begin{center}\begin{equation}\begin{split}#1\end{split}#2\end{equation}\end{center}}
\newcommand{\eqn}[2]{\begin{equation}\begin{split}#1\end{split}\notag#2\end{equation}}
\newcommand{\ee}{\end{equation}}

\newcommand{\ep}{\varepsilon}

\newcommand{\mat}[1]{\begin{pmatrix}#1\end{pmatrix}}
\newcommand{\imply}[1]{\Rightarrow}

\newcommand{\eps}{\varepsilon}
\newcommand{\pd}[2]{\frac{\partial #1}{\partial #2}}

\newcommand{\vt}{\vartheta}
\newcommand{\RM}{\mathbb{R}}
\newcommand{\cn}{\operatorname{cn}}

\newcommand{\bap}{\begin{app}}
\newcommand{\eap}{\end{app}}
\newcommand{\begs}{\begin{exams}}
\newcommand{\eegs}{\end{exams}}
\newcommand{\beg}{\begin{example}}
\newcommand{\eeg}{\end{exaplem}}
\newcommand{\bpr}{\begin{proposition}}
\newcommand{\epr}{\end{proposition}}
\newcommand{\bt}{\begin{theorem}}
\newcommand{\et}{\end{theorem}}
\newcommand{\bc}{\begin{corollary}}
\newcommand{\ec}{\end{corollary}}
\newcommand{\bl}{\begin{lemma}}
\newcommand{\el}{\end{lemma}}
\newcommand{\bd}{\begin{definition}}
\newcommand{\ed}{\end{definition}}
\newcommand{\brs}{\begin{remarks}}
\newcommand{\ers}{\end{remarks}}

\newcommand{\kStableLeft}{0.9426}

\newcommand{\kUnstableLowerLeft}{0.199910210210210}
\newcommand{\kUnstableLowerRight}{0.942197747747748}

\newcommand{\kUnstableUpperRight}{0.999999999997}

\newcommand{\XUnstableLowerLeft}{6.28\ }
\newcommand{\XUnstableLowerRight}{8.44 }

\newcommand{\XUnstableUpperRight}{48.3\ }

\newtheorem{theorem}{Theorem}[section]
\newtheorem{proposition}[theorem]{Proposition}
\newtheorem{corollary}[theorem]{Corollary}
\newtheorem{lemma}[theorem]{Lemma}
\newtheorem{definition}[theorem]{Definition}

\newtheorem{example}[theorem]{Example}
\newtheorem{remark}[theorem]{Remark}

\title{ Numerical proof of stability of roll waves in the small-amplitude limit for inclined thin film flow
}

\author{\sc \small Blake Barker}
\thanks{Indiana University, Bloomington, IN 47405;
bhbarker@indiana.edu: Research of B.B. was partially supported
under NSF grants no. DMS-0300487 and DMS-0801745 and the College of Arts and Sciences Dissertation Year Fellowship.}

\begin{document}

\maketitle


\begin{abstract}
We present a rigorous numerical proof based on interval arithmetic computations categorizing 
the linearized and nonlinear stability of periodic viscous roll waves of the KdV-KS equation modeling weakly unstable flow
of a thin fluid film on an incline in the small-amplitude KdV limit.
The argument proceeds by verification of a stability condition derived by
Bar-Nepomnyashchy and Johnson-Noble-Rodrigues-Zumbrun involving inner products of various elliptic functions arising through the KdV equation.
One  key point in the analysis is a bootstrap argument balancing the extremely poor sup norm bounds for these functions against  the extremely good convergence properties for analytic interpolation in order to obtain a feasible computation time.
Another is the way of handling analytic interpolation in several variables by a two-step process carving up the parameter space into manageable pieces for rigorous evaluation.
These and other general aspects of the analysis should serve as blueprints for more general analyses of spectral stability.
\end{abstract}

\tableofcontents

\section{Introduction}

In this paper we study by a rigorous analytical and numerical investigation the spectral stability of periodic wave train solutions of the 
Korteweg-de Vries--Kuramoto-Sivashinsky equation (KdV-KS),
%
\begin{equation}
u_t + (u^2/2)_x + \eps u_{xxx} + \delta(u_{xx}+u_{xxxx})=0,\quad \eps^2+\delta^2 = 1,\quad t>0,\ x\in \R,
\label{gks}
\end{equation}
in the limit $\xi\neq 0, \ \delta \to 0$. The 
KdV-KS equation 
has been used to model a wide variety of phenomena including pattern formation and hydrodynamic instability \cite{S1,S2}. For $1\gg \delta\sim \sqrt{F-2}>0$, \eqref{gks} can be derived with formal asymptotics from the St. Venant shallow water equations \cite{YY},
$
h_t+(hu)_x=0,\ (hu)_t+(hu^2+h^2/2F^2)_x=h-u^2+\nu(hu_x)_x,
$
as the Froude number $F\to 2^+$, which is significant since constant solutions are unstable for $F >2$. Alternatively, \eqref{gks} may be derived with formal asymptotics from the full Navier-Stokes (NS) equations \cite{W} as the Nusselt number goes to 0, that is for $0<R-R_c\ll 1$, where Nusselt flows are unstable for $R$ greater than the critical Reynolds number $R_c$. In these limits, the period scales as $1/\delta$ and the amplitude as $\delta^2$ so that instabilities are of small-amplitude long wave type.

When $\delta = 0$, \eqref{gks} is the Korteweg-de Vries (KdV) equation which is a Hamiltonian system. In particular, \eqref{gks} is a singular perturbation of the KdV equation.

Analytical and numerical studies \cite{BJNRZ1,BN,CDK,JNRZ} indicate that a band of stable periodic traveling-wave solutions of \eqref{gks} continues from the classical KS limit ($\eps = 0$) to the KdV limit; see Figure \ref{islands}. 

In \cite{BN} the authors find that the stability band in the limit $\delta \to 0$ is given by $[X_l,X_u]$ where 
\begin{equation}
X_l\approx \frac{2\pi}{0.744\pm 0.001}\approx 8.44\pm 0.01\ \mathrm{and} \ X_u \approx \frac{2\pi}{0.239\pm 0.001}\approx 26.29\pm 0.11.
\label{boundaries}
\end{equation}
Similar results were obtained in \cite{BJNRZ1,CDK}; see Figure \ref{islands}.  In the $\delta \to 0$ limit the Evans function computations used in \cite{BJNRZ1} become more demanding due to it being a singular limit and are not readily accessible to direct numerical computation. 
This limit is of particular interest as the one
governing canonical \qut{weakly unstable} behavior \cite{BN,N}, in the
sense that the lowest-order term in the associated perturbation expansion 
(corresponding to the coefficient of the second-order derivative term) vanishes.
\begin{center}
\begin{figure}[htbp] \label{islands}
$\begin{array}{cc}
(a) \includegraphics[scale=.3]{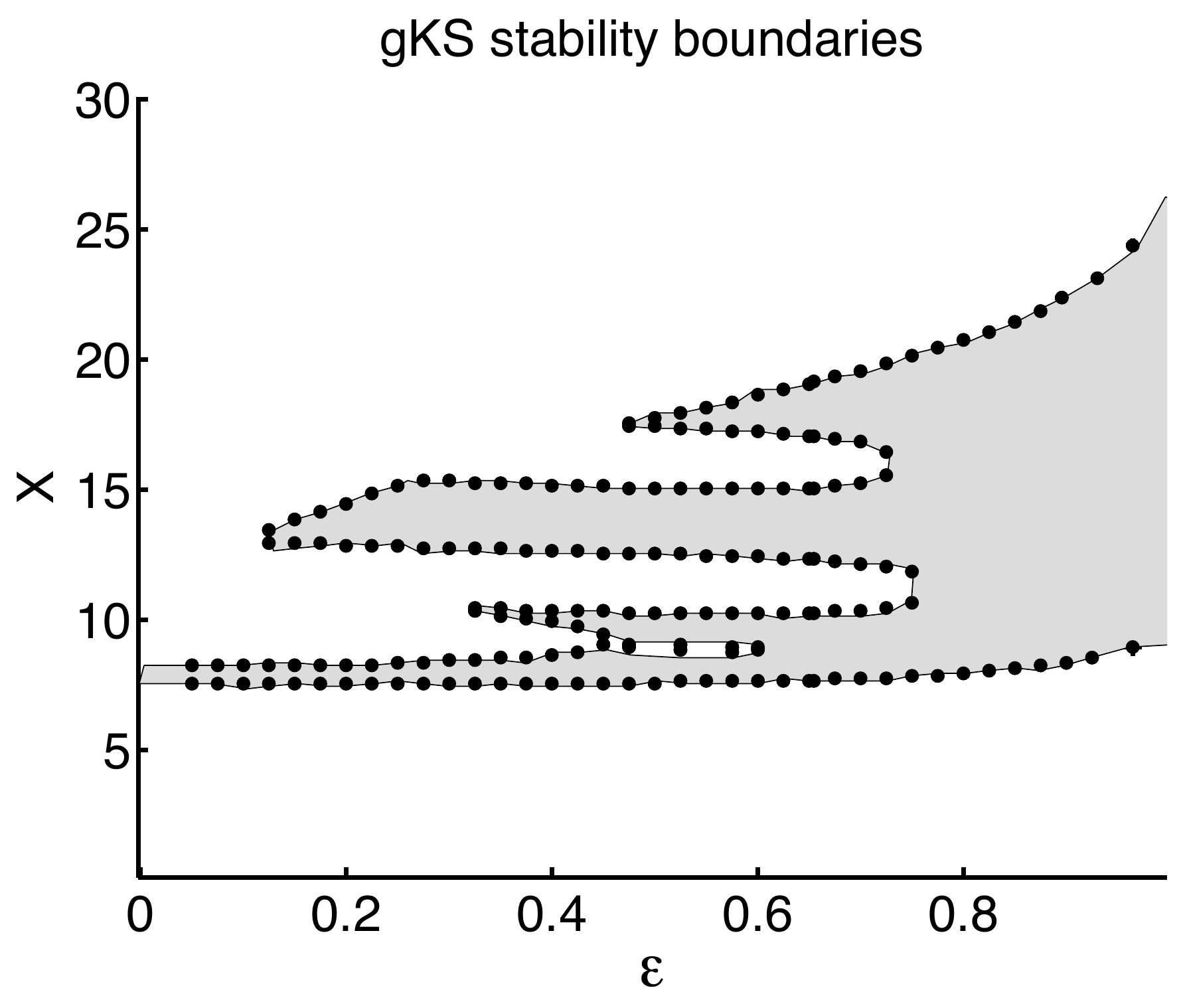} & (b) \includegraphics[scale=0.3]{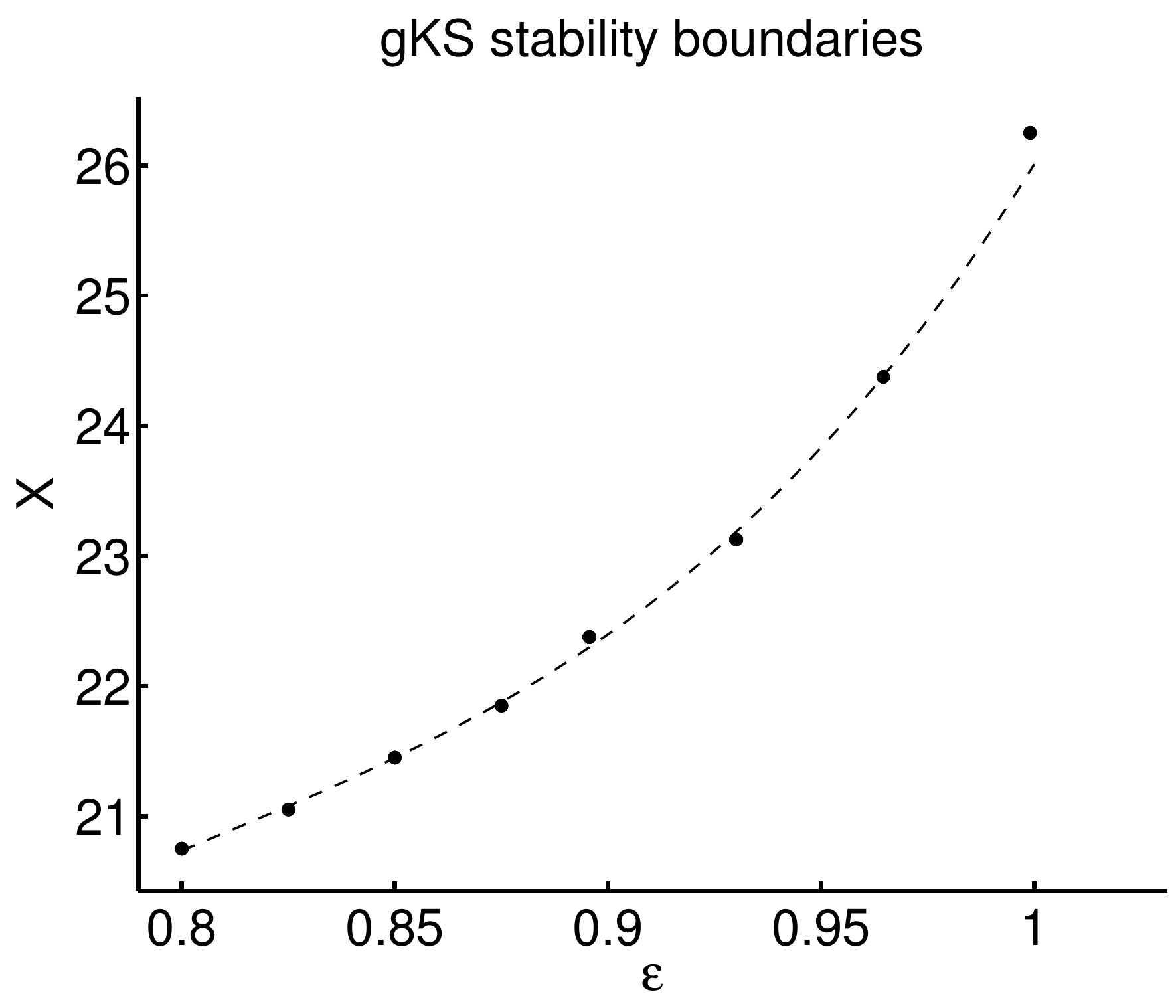}
  \end{array}$
\caption{(a) Reprint from \cite{BJNRZ1}. We plot in gray the period $X$ against $\eps$ corresponding to stable traveling-wave solutions of \eqref{gks}. As $\eps \to 1$, $\delta \to 0$ corresponding to the KdV limit. In this figure, stability was determined using the Evans function which does not involve a singular perturbation. (b) A zoomed in picture of (a). The dashed line plots the best cubic least squares fit for data points $0.8< \eps < 0.98$. The cubic fit predicts a stability transition at $X = 26.01$ when $\eps = 1$. To find the data points, a bisection method was used on the Evans function in the variable $X$ with a relative error bound of $10^{-2}$ as stopping criteria.   }
\end{figure}
\end{center}

\subsection{Background} We begin by reviewing some relevant results.

\subsubsection{Diffusive spectral stability conditions}
Let $u(x,t) = \bar u(x-ct)$ be a spatially periodic solution of \eqref{gks} with period X. By Galilean invariance, we may take $c=0$. Define $F(u):= u_t + (u^2/2)_x+\eps u_{xxx}+\delta(u_{xx}+u_{xxxx})$ and let $u(x,t)= \bar u(x)+ \tilde v(x,t)$ be a solution to \eqref{gks} where $\tilde v(\cdot,t)\in L^2(\R)$. Taking the G\^{a}teaux differential of $F(\cdot)$ at $\bar u$ in the direction $\tilde v(\cdot,\cdot)$ yields the linearized equation $\tilde v_t +(\bar u \tilde v)_x + \eps \tilde v_{xxx} + \delta(\tilde v_{xx}+\tilde v_{xxxx})$. Then substituting the separated solution ansatz $\tilde v(x,t)= e^{\lambda t}v(x)$ into the linearized equation gives the eigenvalue problem, 
\begin{equation}
\lambda v =  -(\bar u v)_x - \eps v_{xxx} - \delta(v_{xx}+v_{xxxx})=0. 
\label{evp}
\end{equation} 
Suppose $\lambda$ is an eigenvalue of \eqref{evp} and $v(\cdot)$ is the corresponding  eigenfunction. Now $v(\cdot)$ has the Fourier transform representation 
\begin{equation}
v(x) = \frac{1}{2\pi}\int_{-\infty}^{\infty} e^{i\xi x}\hat v(\xi)d\xi= \sum_{j\in \Z} \frac{1}{2\pi}\int_{-\pi}^{\pi}e^{i(\xi+2\pi j)x}\hat v(\xi + 2\pi j)d\xi= \frac{1}{2\pi}\int_{-\pi}^{\pi}e^{i\xi x}\hat v(\xi,x)=: w(\xi,x),
\end{equation}
where $\hat v(\xi,x):= \sum_{j\in\Z}e^{i2\pi j x}\hat v(\xi + 2\pi j)$. Note that substitution of $\frac{1}{2\pi}\int_{-\pi}^{\pi}e^{i\xi x}\hat v(\xi,x)d\xi$ into the eigenvalue problem \eqref{evp} leads to the Bloch representation $\hat L_{\xi}[\bar u]:= e^{-i\xi x}L[e^{i\xi x}\bar u]$. Then $\sigma_{L^2(\R)}(L) = \cup_{\xi \in [-\pi/X,\pi/X]} \sigma_{L^2_{\mathrm{per}}}(L_{\xi})$.

The spectral stability conditions for this eigenvalue problem, defined in various contexts  \cite{Sc1,Sc2,JZ1,JZ2,JZN,JNRZ1,BJNRZ2}, are given by:
\begin{enumerate}
  \item[{(D1)}]
$\sigma(L)\subset\{\lambda\ |\ \Re \lambda<0\}\cup\{0\}$.
  \item[{(D2)}]
$\exists \ \theta>0$ such that
$\sigma(L_{\xi})\subset\{\lambda\ |\ \Re \lambda\leq-\theta|\xi|^2\}$ $\forall \xi\in[-\pi/X,\pi/X]$.
  \item[{(D3)}]
$\lambda=0$ is an eigenvalue of $L_0$ with a three dimensional generalized eigenspace
$\Sigma_0$.
\end{enumerate}
The following technical hypothesis is also needed:
\begin{enumerate}
  \item[{(H1)}]
For small $|\xi|$, the three zero eigenvalues of $L_0$ satisfy
$\lambda_j(\xi)= \alpha_j \xi + o(\xi)$ with the $\alpha_j$ distinct.
\end{enumerate}
%
%
Conditions (D1)-(D3) and assumption (H1) for \eqref{gks}, imply the following nonlinear result.

\begin{proposition}[\cite{JZ1,JZN,BJNRZ1,JNRZ2} Nonlinear modulational stability, at Gaussian rate.]{
Assume that conditions (D1)-(D3) and assumption (H1) hold. Then
\begin{equation}
\label{gauss}
\|\tilde u(x,t)-\bar u(x-\psi(x,t))\|_{L^p\cap H^s},
\quad
\| \nabla_{x,t}\psi \|_{W^{s,p}}\leq C(1+t)^{-\frac 12 (1-1/p)},
\qquad 2\leq p\leq \infty,
\end{equation}
}
for localized initial perturbations 
$\| (\tilde u-\bar u)_{t=0} \|_{L^1\cap H^s}$ sufficiently small, 
with $s$ sufficiently large.
\end{proposition}

\subsubsection{The KdV limit}\label{s:limit}
In this paper, we investigate stability of the periodic traveling-wave solutions of \eqref{gks} in the KdV limit, $\delta \to 0$. We begin by stating the known existence result as summarized in \cite{JNRZ}.
\begin{proposition}[\cite{EMR} Existence]\label{p:kdvsolnexpand}
Given any positive integer $r\geq 1$, there exists $\delta_0>0$ such that the periodic traveling wave solutions $u_\delta(\theta)$, $\theta=x-ct$, of \eqref{gks}
are analytic functions of $\theta\in\RM$ and $C^r$ functions of $\delta\in[0,\delta_0)$.  
For $r\geq 3$,
profiles $u_\delta$ expand (up to translation) as
$\delta\to 0$ as a $2$-parameter family
\begin{equation}\label{kdvsolnexpand}
\left\{\begin{aligned}
u_\delta(\theta;a_0,k)&=u_0(\kappa \theta; a_0,k,\kappa)+\delta U_1(\theta)+\delta^2 U_2(\theta)+\mathcal{O}(\delta^3),\\
c&=c_0(a_0,k,\kappa)+\delta^2 c_2+\mathcal{O}(\delta^3),
\end{aligned}\right.
\end{equation}
where 
$$
\displaystyle
u_0(y;a_0,k, \kappa)=a_0+3k\left(\frac{\kappa K(k)}{\pi}\right)^2\cn^2\left(\frac{ K(k)}{\pi}y,k\right),\quad c_0=a_0+(2k-1)\left(\frac{\kappa K(k)}{\pi}\right)^2;
$$
comprise the $3$-parameter family (up to translation) 
of periodic (KdV) profiles and their speeds;
$\cn(\cdot,k)$ is the Jacobi elliptic cosine function
with elliptic modulus $k\in[0, 1)$;
$K(k)$ and $E(k)$ are
the complete elliptic integrals of the first and second kind;
$a_0$ is a parameter related to Galilean invariance;
$k$ is a parameter in one-to-one correspondence with period;
and $\kappa=\mathcal{G}(k)$ is determined via the selection principle 
$$
\displaystyle
\left(\frac{K(k)\mathcal{G}(k)}{\pi}\right)^2=
   \frac{7}{20}\frac{2(k^4-k^2+1)E(k)-(1-k^2)(2-k^2)K(k)}{(-2+3k^2+3k^4-2k^6)E(k)+(k^6+k^4-4k^2+2)K(k)}.
$$
Moreover the functions $(U_i)_{i=1,2}$ are (respectively odd and even) solutions of the linear equations
\begin{equation}\label{cork}
\mathcal{L}_0[U_1]+\kappa u_0''+\kappa^3 u_0''''=0,\quad
\mathcal{L}_0[U_2]+\left(\frac{U_1^2}{2}-c_2 u_0\right)'+\kappa U_1''+\kappa^3 U_1''''=0,
\end{equation}
on $(0,2K(k))$ with periodic boundary conditions,
where 
$\mathcal{L}_0:=\kappa^2 \partial_x^3+\partial_x\left((u_0-c_0)\right)$.

\end{proposition}

We may parameterize waves by $k$ alone since the periodic solutions \eqref{cork} are independent of $a_0$ due to Galilean invariance of \eqref{gks}. 

We introduce two technical hypotheses found in \cite{JNRZ2},

\begin{enumerate}
  \item[{(A1)}]
The non-zero eigenvalues of the linearized (Bloch) KdV operator
$\mathcal{L}_\xi[u_0]$ about $u_0(\cdot; a_0,k, \mathcal{G}(k))$
are simple for each $\xi\in[-\pi/X,\pi/X)$ and
$\lambda=0$ is an eigenvalue only if $\xi=0$.
  \item[{(A2)}]
The zero eigenvalues of $\mathcal{L}_0[u_0]$ expand about $\xi=0$ as
$\lambda_j(\xi)=i\xi \alpha_j + O(\xi^2)$ with $\alpha_j$ distinct.
\end{enumerate}

Under assumption (A1), the spectra $\lambda$ of $\mathcal{L}_\xi$,
for $(\xi,\lambda) \ne (0,0)$, expand formally in $\xi$ \cite{BN} as
\begin{equation}\label{exp_kdvks}
\lambda(\xi)=\lambda_{KdV}(\xi)+\delta\lambda_1(\xi)+O(\delta^2), 
\end{equation}
where $\xi\in(-\pi/X, \pi/X)$ is the Bloch number, 
and $X$ the spatial period of the wave train.

Define the stability condition (S1), $$
\Re(\lambda_1) <0$$ for all $(\xi,\lambda)\neq (0,0)$ for $\lambda_1$ as
in \eqref{exp_kdvks}, where 
\begin{itemize}
\item[{(S1)}]\begin{equation}\label{lambda1}
\Re(\lambda_1)=
\frac{\langle v', v''+v''''\rangle }
{\langle v',v \rangle},
\end{equation}
where $\langle \cdot,\cdot\rangle$ denotes complex $L^2_{per}(0,X)$ inner
product, and $v$ is the antiderivative of an 
associated eigenfunction of $\mathcal{L}_\xi$,
which is explicitly computable in terms of Jacobi elliptic functions; 
see \cite{BN} or Appendix A.1 of \cite{BJNRZ2}. For a definition of $v(\cdot)$, see equation \eqref{vdef}. 
\end{itemize} 

The following theorem reduces the question of stability to a few simple conditions to be verified.

\begin{proposition}[\cite{JNRZ2} Limiting stability conditions ]\label{p:kdvstab}
For fixed $(a_0,k)$, let $u_\delta(\cdot;a_0,k)$ 
denote a family of roll-wave solutions \eqref{kdvsolnexpand}
of \eqref{gks} as $\delta\to 0$, and let $\Re \lambda_1$ be defined as
in \eqref{lambda1} for all $\xi\neq 0$, $\lambda_{KdV}$
(or, in case (A1) fails, by
continuous extension via the explicit parametrization of \cite{BN}).
(i) If $\Re \lambda_1>0$ for some $\xi\neq 0$, $\lambda_{KdV}$, then 
$u_\delta(\cdot;a_0,k)$ is spectrally \emph{unstable} for $\delta$ sufficiently
small and nearby $(a_0,k)$ (equivalently, nearby limiting period $X$).
(ii) If $\Re \lambda_1<0$ for all $\xi\neq 0$ and (A1)-(A2) are satisfied
for the limiting wave $u_0$, then
$u_\delta(\cdot;a_0,k)$ is spectrally (hence nonlinearly)
\emph {stable} for $\delta$ sufficiently small and nearby $(a_0,k)$ (equivalently, nearby limiting period $X$).
\end{proposition}

Numerical computations of \cite{BN,BJNRZ2} indicate that the
limiting stability conditions hold together with (A1)-(A2)
for limiting periods $X$ in an interval $(X_m,X_M)$,
and fail for $X$ outside $[X_m,X_M]$; see equation \eqref{boundaries}.

\subsection{Description of the main result}\label{mainresult}

Our present purpose is to rigorously verify the numerical observations of \cite{BN,JNRZ} that stability occurs on a limiting interval
$[X_m,X_M]$ as $\delta \to 0$ by verifying that assumptions (A1)-(A2) and (S1) hold for period $X\in [X_l,X_r]\subset [X_m,X_M]$, implying that for $\delta>0$ sufficiently small, $X$ periodic waves of \eqref{gks} are spectrally, hence nonlinearly, stable by Proposition \ref{p:kdvstab}. The main contribution here is completely rigorous numerical verification of stability of a family of periodic traveling waves of \eqref{gks} in the limit $\delta \to 0$.

Our main theorem, proven by interval arithmetic, is as follows:

\begin{theorem}\label{thm:main}
There are $k_l \in [0.9421,0.9426]$, $k_r\in[0.99999838520,0.99999838527]$, corresponding to $X_l\in [8.43,8.45]$ and $X_r\in [26.0573, 26.0575]$, and $k_{\min} =$ \kUnstableLowerLeft \   and $k_{\max}=$ \kUnstableUpperRight, corresponding to $X_{\min}\approx 6.28$ and $X_{\max}\approx 48.3$, such that the periodic traveling-wave solutions of \eqref{gks} described in Proposition \ref{p:kdvsolnexpand} are spectrally unstable  on $[k_{\min},k_l]$, corresponding to $[X_{\min},X_l]$ , and $[k_r,k_{\max}]$, corresponding to $[X_r,X_{\max}]$, and are spectrally, thus nonlinearly, stable on $k\in[k_l,k_r]$, corresponding to $[X_l,X_r]$.
\end{theorem}

We note that the limits $k\to 0$ and $k\to 1$ are not accessible numerically with the approach of this paper, but should be treatable by asymptotic analysis. The limit $k\to 0$ corresponds to the limiting Hopf bifurcation, and as $k\to 1$, the periodic profiles converge to the limiting homoclinic solution; See Figure \ref{fig171}. 

\begin{figure}[htbp]
 \begin{center}
$
\begin{array}{lcr}
(a)\includegraphics[scale=0.3]{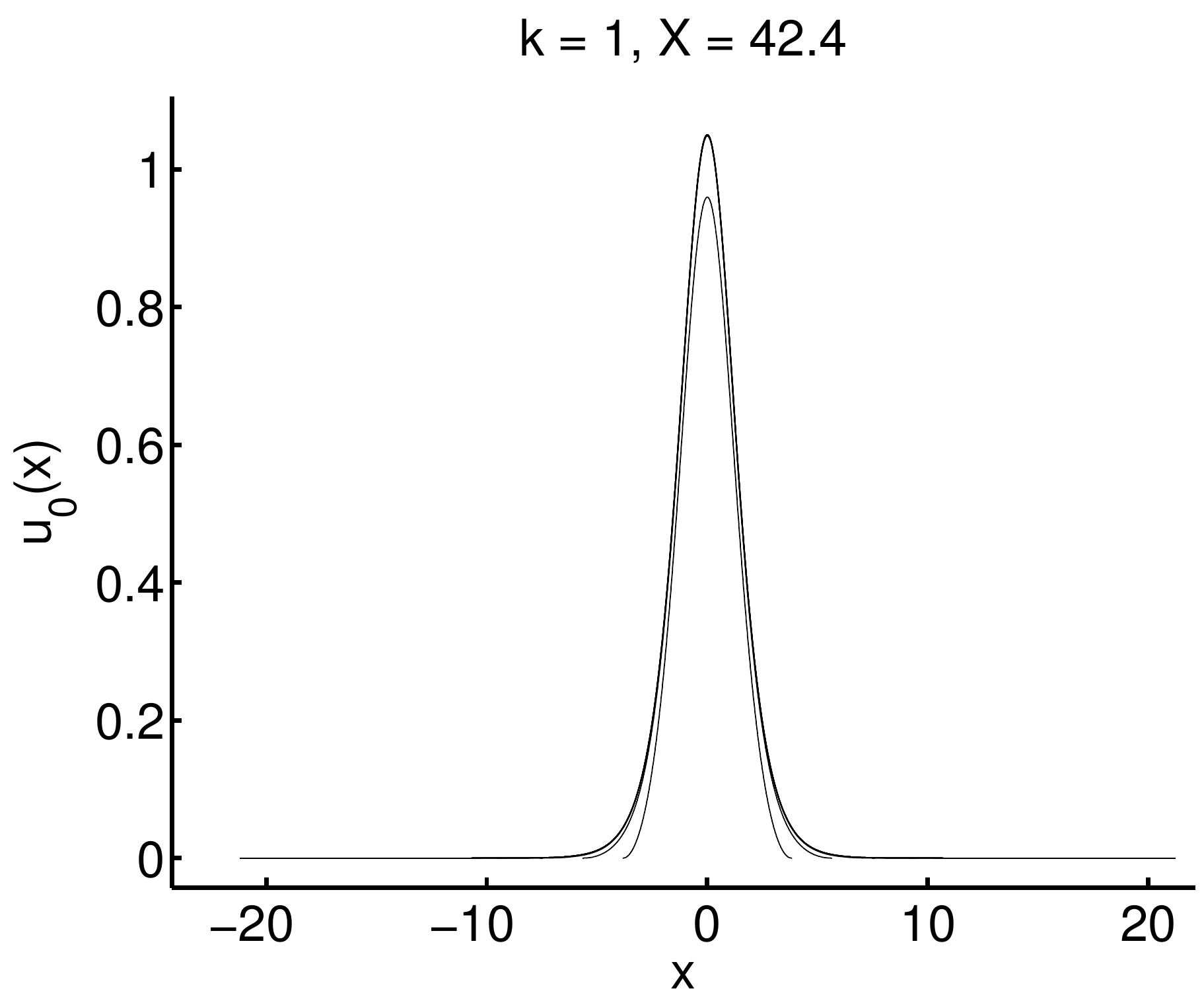}&(b)\includegraphics[scale=0.3]{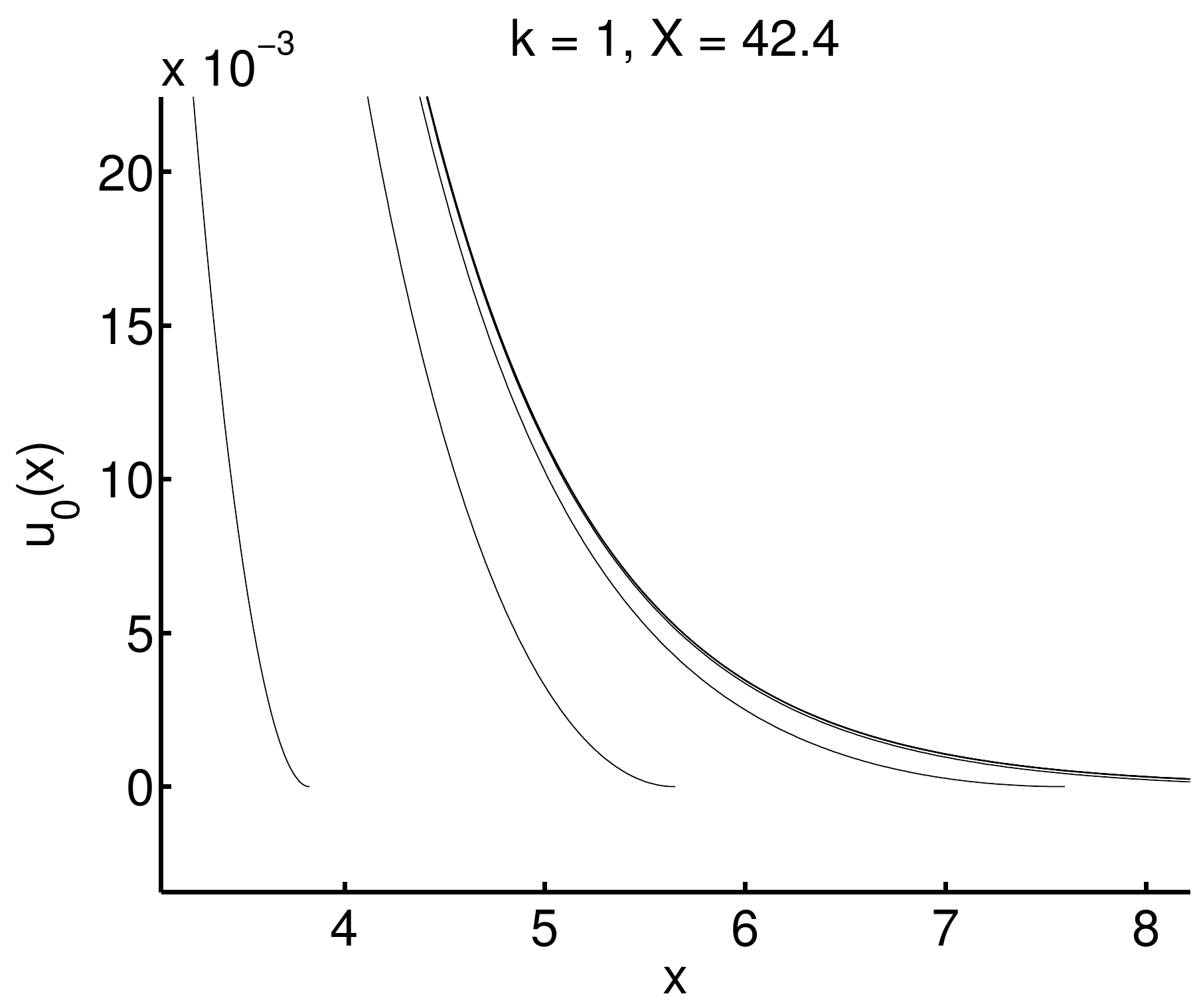}
\end{array}
$
\end{center}
\caption{(a) We plot $u_0(xK(k)/X)$ verse $x$ for several periods $X$ where $u_0(\cdot)$ is described in Proposition \ref{p:kdvsolnexpand}. (b) We zoom in on (a) showing convergence of the periodic to a pulse. }
\label{fig171}\end{figure}

\subsection{Discussion and open problems}

This gives rigorous validation for the first time of any roll wave solution of a conservation law. The associated analysis is delicate since the spectra of the limiting KdV waves is completely neutral.
An interesting related issue is that
solitary waves, despite having unstable essential spectrum, appear to dominate asymptotic behavior of stability of weakly unstable thin film flow \cite{PSU}. 
The mathematical explanation for this puzzling phenomenon is that
near-solitary waves can coexist in mutually stabilizing near-periodic arrays \cite{JNRZ2}. A heuristic explanation of this mutual stabilization is found in \cite{BJNRZ2,BJNRZ5}. The present work provides rigorous verification of stability of near-solitary periodic waves, confirming the spectral stability assumptions made in \cite{JNRZ2} and supported by nonrigorous numerics in \cite{BJNRZ1,BJNRZ2,JNRZ1}.
For physical background, see \cite{BN,PSU,JNRZ}. 
The present work has separate mathematical interest as a perturbed integrable system analysis.  

We find that the lower stability boundary occurs for $X \in[8.43,8.45]$ which agrees with the value $X\approx \frac{2\pi}{0.744\pm 0.001}\approx 8.445\pm 0.011$ determined in \cite{BN}. However, we find that the upper stability boundary occurs for $X\in [26.0573, 26.0575]$ which slightly differs from
the value
 $X \approx \frac{2\pi}{0.239\pm 0.001}\approx 26.29\pm 0.11$ reported in\cite{BN}. That is,  in \cite{BN}  the computationally difficult  upper stability boundary is accurate to only two digits even though the computations used double precision arithmetic,
whereas we, by an additional post-processing step whose necessity is indicated by our interval bounds, achieve accuracy up to five
digits, and in principle more.
 This demonstrates the additional advantage that interval arithmetic and rigorous verification may provide. In general, the techniques of this paper may provide a useful guide for future analysis involving rigorous verification using one and two-dimensional analytic interpolation and bootstrapping techniques, subdivision of domains to keep the number of interpolation nodes small, verification of strict stability transition, and the interval evaluation of a polynomial interpolant with Taylor expansion to reduce the width of the resulting interval.

A future direction would be to establish stability or spectral instability in the $\delta \to 0$ limit for the entire family of periodic waves of \eqref{gks} given by \eqref{kdvsolnexpand} including the homoclinic or infinite-period limit which will require different analysis, an avenue we plan on pursuing.
A long-term goal is to build functionality for automatic convergence error estimation into STABLAB \cite{BHZ2}, a MATLAB based numerical library for the study of traveling waves, with interval arithmetic making Evans function computations completely rigorous. 
Obtaining this goal would complete the program proposed in \cite{ZH} by Zumbrun and Howard to determine stability of general traveling waves.

\subsection{Protocol and readers guide}

In this section we describe the protocol we follow to provide numerical proof, and explain what we mean by interval arithmetic and numerical proof.

\subsubsection{Interval arithmetic}
In scientific computation, one seeks to use algorithms that increase computational speed and control round-off error. Many examples, such as the Patriot Missile failure in Dhahran in 1991 or the explosion of the Ariane 5 (flight 501) rocket in 1996 \cite{IB}, demonstrate the potential consequences of approximating real number operations with machine arithmetic. Arbitrary precision arithmetic can reduce the size of approximation error, but ultimately, error exists when representing the real numbers with a finite subset. Interval arithmetic bounds round-off error by utilizing intervals that contain the numbers of interest. Operations may then be defined on these intervals. For example, if $A$ and $B$ are two intervals, then an interval operation, as carried out by the computer, is defined by $A\cdot B\to C$ where $C$ is an interval such that $a\cdot b \in C$ whenever $a\in A$ and $b\in B$. Preferably, $C$ is the smallest interval with this property. 

Many software packages are available for interval arithmetic computation. For example, MATHEMATICA supports interval arithmetic directly. The package intpakX provides interval arithmetic support for MAPLE, and INTLAB does the same for MATLAB. The Boost project has an interval arithmetic implementation for C++, and Python has support through the package SymPy. The level of development of these packages varies from experimental to highly developed. We use the MATLAB based package INTLAB \cite{R} which has support for complex interval arithmetic. 

\subsubsection{What is numerical proof?} Various standards of numerical rigor exist in mathematical literature. In the present work we are interested in computer assisted proof. Numerical proof inherently includes an empirical component since the validity of computations rely on the computer hardware and software functioning as supposed at run time. However, if carefully constructed and auditable, numerical proof provides a compelling argument that a result is true.    

In the present work we consider a theorem to be established via numerical proof when we have carried out a computation with known error bounds that implies the theorem is true and we have provided sufficient details to make independent verification reasonably accessible. In particular, by providing sufficient details, we mean that algorithms and their error bounds are described in the paper, computations employ interval arithmetic to bound machine truncation error, the source code is available somewhere accessible, and the computational details are provided. Documentation for the source code is available at \cite{Bdoc} and source code is available upon request. By using INTLAB in MATLAB and providing the source code, we consider our study to be reasonably verifiable. It is also reproducible \cite{Rep} in the sense that the source code is provided along with the details describing the computing environment at run time. 

\subsection{Plan of the paper}

In Section \ref{stabsingle}, for clarity,  we carry out first in full detail a proof of stability of a single wave, that is, for a single value of $k$. In particular, we verify each of the conditions (A1), (A2), and (S1). In Section \ref{stabmid} we verify stability for $k\in[$\kStableLeft,0.9999983], similar to how we did for a single wave. In Sections \ref{lowunstab} and \ref{upunstab} we verify condition (S1) implies spectral instability for waves corresponding to $k\in $[\kUnstableLowerLeft,\kUnstableLowerRight] and [0.99999839,\kUnstableUpperRight] respectively. Then in Section \ref{strict} we determine that the stability transitions are sharp, and in Section \ref{mainproof}, combining all of the previous results, we give the proof of the main theorem. 


\section{Chebyshev interpolation of analytic functions}\label{chebyshev}

A big part of our strategy will be to use favorable properties of analytic functions to greatly reduce the amount of time needed to compute the stability condition (S1). Analytic interpolation allows us to closely approximate the functions involved in the stability condition with a small number of interpolation nodes. We only need provide a very rough bound on the modulus of the function in a small region in order to make this strategy work.  In this section we provide details for this strategy beginning with a brief review of Chebyshev interpolation of analytic functions; see for example \cite{DY,BBCL,RW,T}.
\subsection{One dimensional interpolation}
Let $f(z)$ be analytic inside and on the stadium 
$
E_{\rho}:= \left\{z\in \C| z = \frac{1}{2}\left(\rho e^{i\theta}+e^{-i\theta}/\rho\right), \theta \in [0,2\pi]\right\},
$
also known as Bernstein's regularity ellipse, where $\rho > 1$; see Figure \ref{fig-stadium}.
Let $p_N(x) = \sum_{n=0}^{N}c_nT_n(x)$ be the interpolating polynomial
of degree $N$ of $f(x)$ with interpolation nodes at the extremal points $x_j = \cos\left(j\pi/N\right)$ or the zeros $x_j = \cos\left(2(j+1)\pi/2(N+1)\right)$ of $T_{N+1}(x)$, where $T_n(x)$ is the nth degree Chebyshev polynomial:
$T_0(x) := 1,$ $T_1(x) := x$, $T_{n+1}(x) := 2xT_n(x)-T_{n-1}(x)$, $(n\geq 2)$.
Define
$
W_{N+1}(z):= 
 (z-x_0)(z-x_1)...(z-x_n).
$
%
%
\begin{figure}
\begin{center}
\includegraphics[scale=0.4]{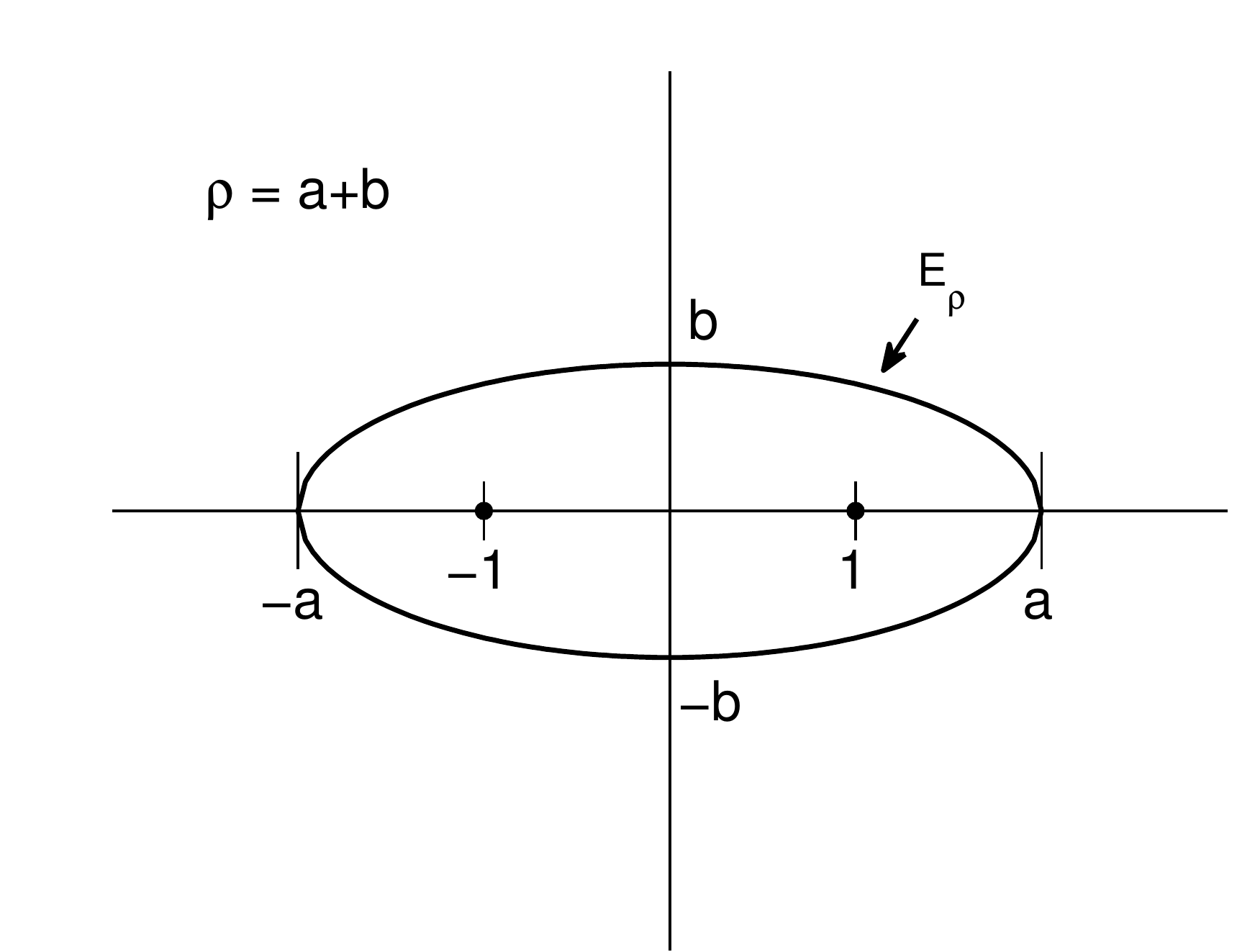}
\end{center}
\caption{Plot of the stadium $E_{\rho}$. }
\label{fig-stadium}
\end{figure}
By Hermite's formula,
\begin{equation}
f(x)-p_N(x) = (2\pi i)^{-1}\int_{E_{\rho}} (W_{N+1}(x)f(z))/(W_{N+1}(z)(z-x))dz
\end{equation}
for $x\in \I:=[-1,1]$. Following \cite{RW}, we have for $x\in \I$ that 
\begin{equation}
\begin{split}
\left|f(x)-p_N(x)\right|&\leq M_{\rho}L_{\rho}(2\pi D_{\rho} \sinh(\eta)\sinh(\eta N))^{-1}
\end{split}
\label{interp-bound1}
\end{equation}
if $x_j = \cos\left(j\pi/N\right)$ and
\begin{equation}
\begin{split}
\left|f(x)-p_N(x)\right|&\leq M_{\rho}L_{\rho}(\pi D_{\rho} \sinh(\eta (N+1)))^{-1}
\end{split}
\label{interp-bound2}
\end{equation}
if $x_j = \cos\left(2(j+1)\pi/2(N+1)\right)$
where 
\begin{equation}\label{interpstuff}
\eta:= \log(\rho),\quad
D_{\rho}:= \frac{1}{2}(\rho + \rho^{-1})-1, \quad
L_p:= \pi \sqrt{\rho^2+\rho^{-2}},\quad
M_{\rho}: = \max_{z\in E_{\rho}}(|f(z)|).
\end{equation}
Here $D_{\rho}$ is a lower bound on the distance of  the stadium $E_{\rho}$ to the line segment $[-1,1]$, and $L_{\rho}$ is an upper bound on the length of $E_{\rho}$.
If $x_j = \cos\left(j\pi/N\right)$, then \cite{RW}
$
2\sinh(\eta)\sinh(\eta N) \leq |W_{N+1}(z)|\leq 2\cosh(\eta)\cosh(\eta N),
$
and if $x_j = \cos\left(2(j+1)\pi/2(N+1)\right)$, then
$
\sinh(\eta(N+1))\leq |W_{N+1}|\leq \cosh(\eta(N+1)).
$
Note that a crude bound $M_{\rho}$, which can be computed with interval arithmetic, still results in exponential decay of error.

\subsection{Two dimensional interpolation}
We now consider interpolation in two dimensions.
Following \cite{Kh} let $\mathbb{P}_N$, $N\in \N$, denote the space of polynomials over $\C$ with degree 
 $\leq N$. Let $\ep = \{\ep_0,\ep_1,...,\ep_N | \ep_j<\ep_{j+1},\ \ep_j \in \I\}$ be interpolation nodes, $f:\I\to \C$, and $L_{\ep,N}$ be the interpolation operator
 defined by $L_{\ep,N}f(\ep_j) = f(\ep_j)$, $L_{\ep,N}f\in \mathbb{P}_N$, and take as norm $||\cdot || = ||\cdot||_{\infty}$.
 Suppose $\hat p\in \mathbb{P}_N$ minimizes $||f-p||$. Note that $\hat p = L_{\ep,N}\hat p$ and thus
 \begin{equation}
 \begin{split}
 ||f-L_{\ep,N}f||&= ||f-\hat p + \hat p - L_{\ep,N}f||
 \leq ||f-\hat p|| + ||L_{\ep,N}\hat p - L_{\ep,N}f||\\
 &\leq ||f-\hat p|| + ||L_{\ep,N}(\hat p - f)||
  \leq  ||f-\hat p|| + ||L_{\ep,N}||||f-\hat p||
 = (1+||L_{\ep,N}||)||f-\hat p||.
\end{split}
\notag
\end{equation} 
The Lebesgue constant is defined as $\Lambda_{\ep,N}:= ||L_{\ep,N}||$. For the Chebyshev polynomials of the first kind, the Lebesgue constant is given by 
$
 \Lambda_{N-1} = \frac{2}{\pi} \log(N) + \frac{2}{\pi}(\gamma + \log(8/\pi)+\alpha_N,\quad 0<\alpha_N<\frac{\pi}{72N^2},
$
where $\gamma = 0.5772...$ is Euler's constant; see \cite{BBCL,Gu}. If the operator $L$ corresponds to tensor product interpolation on the Chebyshev zeros of $T_{N+1}$ in $n$ dimensions, then $||L||_{\infty} = \prod_{j=1}^n$; see \cite{M}.
 
Now let $f:\I^2\to \C$ and define $L_x$ and $L_y$ to be respectively the interpolation operators in the $x$ and $y$ coordinates. For example, $L_xf(x,y) = f(x,y)$ for $y\in \I$ and $x\in \ep$ where for fixed $y$, $L_xf(x,y)$ is a polynomial in $x$ with coefficients $c_j(y)$ and $\ep$ is the set of interpolation nodes for $x$.  
In the two-dimensional case, we let $Lf$ be the polynomial in $\mathbb{P}_{N_1}\times \mathbb{P}_{N_2}$ that equals $f$ on $\ep_x\times \ep_y$. Note that 
$Lf = L_x(L_y f)$. We have the error bounds,
\begin{equation}
\begin{split}
||f-Lf||&= ||f-L_x(L_yf)||
= || f-L_xf + L_xf -L_x(L_yf)||\\
&\leq || f-L_xf ||+|| L_xf -L_x(L_yf)||
= || f-L_xf ||+|| L_x(f -L_yf)||\\
&\leq || f-L_xf ||+ \Lambda_{\ep_x,N_x}||f -L_yf||.
\end{split}
\end{equation}
\subsection{Derivatives of the interpolant}\label{dererror} The $k$th derivative of the interpolant, $p_N(x)$, can be used to approximate the $k$th derivative of $f$. From Hermite's formula,
\begin{equation}
\begin{split}
f^{(k)}(x)-p_N^{(k)}(x) &= \frac{1}{2\pi i}\int_{\Gamma} \pd{^k}{x^k}q(x) \frac{f(z)}{w_{N+1}(z)}dz,
\end{split}
\end{equation}
where $q(x):= w_{N+1}(x)(z-x)^{-1}$.
When the interpolation nodes are the Chebyshev zeros, $w_{N+1}(x) = T_{N+1}(x)$, and $w_{N+1}'(x) = (N+1)U_n(x)$. If $N$ is odd, $U_N(x) = 2\sum_{j \mod2 \equiv 1}$ and if $N$ is even, $U_n(x) = -1 + 2\sum_{j\mod 2 \equiv 0}$. Thus, $|w_{N+1}'(x)| \leq (N+1)(N+3)$. Note that $T_{N+1}''(x) = (n+1)\left( (n+2)T_{n+1}-U_{n+1}\right)/(x^2-1)$, which is bounded by $2(n+1)(n+3)/(x^2-1)$.
Hence, recalling the definitions given in \eqref{interpstuff}, the error when using the Chebyshev zeros for the interpolation nodes is given by,
\begin{equation}\label{eq:derbounds}
\begin{split}
|f'(x)-p_N'(x)|_{\infty, x\in[-1,1]} &\leq \frac{LM}{2\pi \sinh(\eta (N+1))}\left(\frac{(N+1)(N+3)}{D_{\rho}}+\frac{1}{D^2_{\rho}}\right),\\
|f''(x)-p_N''(x)|_{\infty, x\in[-1,1]} &\leq \frac{LM}{2\pi \sinh(\eta (N+1))}\left(\frac{2(N+1)(N+3)}{D_{\rho}(x^2-1)}+\frac{(N+1)(N+3)}{D_{\rho}^2}+\frac{1}{D^3_{\rho}}\right).
\end{split}
\end{equation}

Now suppose that $p(x,y)$ is the interpolant of $f(x,y)$ on $\I^2$ and that $D_k$ is the operator that takes the $k$th derivative with respect to $x$. Then
\begin{equation}
\begin{split}
|D_k f -D_k L f|&=|D_kf-D_kL_yL_xf|
\leq ||D_kf-D_kL_yf||+||D_kL_yf-D_kL_yL_xf||\\
&\leq ||D_k f- L_y(D_kf)||+||L_y||\ ||D_kf-D_kL_xf||.
\end{split}
\end{equation}
Cauchy's integral formula may be used to bound $D_k f$ on a stadium of smaller radius than the one on which $f$ is bounded.
\subsection{Computing the coefficients of the interpolant}
The standard choice of algorithm to obtain the interpolation coefficients is the fast cosine transform. However, with interval arithmetic the speed of algorithms is greatly affected by the cost of switching the rounding mode. In our study, we found that using INTLAB's fast interval matrix multiplication with vectorization may be faster.
In the one dimensional case, the interpolation coefficients are given by \cite{Ba},\\
$$
a_0:= \frac{1}{N+1}\sum_{r= 0}^{N} f(x_r)T_0(x_r),\quad 
a_j := \frac{2}{N+1}\sum_{r=0}^N f(x_r)T_j(x_r),
$$
for $j = 1,..,N$
or by 
$a_0 = \frac{1}{2} V_1$, and $a_j = V_{j+1}$ for $j> 0$ where 
\eqn{
V:= 2 \mat{f(x_0)&f(x_1)&...&f(x_N)}[T_{k}(x_j)]_{k,j=0,1,...,N}.
}{}
In the two dimensional case the coefficients are given by $a_{0,0} = V_{0,0}/4$, $a_{j,0} = V_{j,0}/2$ for $j = 1,...,M$, $a_{0,k} = V_{0,k}/2$ for $k = 1,...,N$, and $a_{j,k} = V_{j,k}$ for $j = 1,...,M$, $k = 1,...,N$ where  
\begin{equation}
V := \frac{4}{(M+1)(N+1)}[T_{k}(x_j)]_{k,j = 0,...,M}[f(x_j,y_q)]_{j = 0,...,M; q = 0,...,N}[T_{r}(y_q)]_{q,r = 0,...,N}.
\end{equation}

\subsection{Evaluating the interpolant}\label{evaluate}
The Chebyshev polynomials satisfy $T_n(x) = \cos(n\cos^{-1}(x))$ so that we may evaluate the polynomial interpolant $p_N(x)$ in terms of $\theta$, $p_N(\theta) = \sum_{n=0}^Nc_n\cos(n\theta)$, where $x = \cos(\theta)$. Evaluating the two-dimensional polynomial interpolant
\begin{equation}
P_{N,M}(x,y)=\sum_{n=0}^N\sum_{m=0}^M c_{nm}T_N(x)T_m(y)=\sum_{n=0}^N\sum_{m=0}^M c_{nm}\cos(n\theta)\cos(m\nu)
\notag
\end{equation}
with interval arithmetic sometimes yields a poor result since $(N+1)\times(M+1)$ intervals must be added. We significantly improve this by Taylor expanding $P_{N,M}$ in the variables $\theta$ and $\nu$ where $x = \cos(\theta)$ and $y = \cos(\nu)$. Taylor expanding to fifth order allows us to compute $p_{N,M}$ and its partial derivatives up to 5th order with small intervals representing a single point in the interval of interest. To obtain an interval representation of the Taylor remainder, we must evaluate $p_{N,M}$ on the full intervals in $\theta$ and $\nu$, but the contribution of the remainder term to the interval width is not significant when the intervals in $\theta$ and $\nu$ are small.

We are also interested in evaluating the integral of the one dimensional interpolant. 
Note that under the transformation $x = \cos(\theta)$,
\begin{equation}
\begin{split}
 \int_{-1}^1 T_N(x) dx & = \int_0^{\pi/2} \cos(n\theta)\sin(\theta)d\theta + \int_{\pi/2}^{\pi}\cos(n\theta)\sin(\theta)d\theta
 = \frac{1+(-1)^n}{1-n^2}.
\end{split}
\end{equation}

 \textbf{Computational detail 1.}
  In practice, we only compute the even indexed interpolation coefficients when integrating since $\int_{-1}^1T_n(x)dx = 0$ for $n$ odd. 
 
 \textbf{Computational detail 2.}
Clenshaw's method is often used for efficient evaluation of a Chebyshev polynomial.  However, this method results in wide intervals when using interval arithmetic. Indeed, if the highest order coefficient of the interpolating polynomial  has an interval error bound of width $\varepsilon>0$ , then by the termination of the Clenshaw algorithm, the width of the interval answer is at least $(2|x|)^{N-2}|x|\eps$. If $x = 1$, and $\eps = 2^{-52}$, then the error interval for $N = 106$ is at least $\frac{1}{\eps}= 2^{52}$.  This demonstrates the unique challenge interval arithmetic can pose. On the other hand, by using the property $T_n(x) = \cos(n\cos^{-1}(x))$ to evaluate the finite Chebyshev series, the interval error $\eps_j$ for the $jth$ coefficient $c_j$ contributes error of at most  $|T_j(x)|\eps_j \leq \eps_j$ and so the interpolation error grows at most only linearly with $N$. As described above, Taylor expanding yields further improvement when evaluating the Chebyshev polynomial on a larger interval.

\textbf{Computational detail 3.}
When carrying out analytic interpolation to approximate a function $f$, it is sometimes advantageous to break the domain up into smaller sub-domains when the domain comes close to a pole of $f$. This increases how large we may take $\rho$ which plays a significant role in minimizing the error terms \ref{interp-bound1} and \ref{interp-bound2}.


\section{Stability of a single wave}\label{stabsingle}

In this section we show that assumptions (A1) and (A2) hold and that $\Re(\lambda_1) < 0$ for a single wave $X(k) \in  [  11.30108911018488,  11.30108911018549]$ for $k = 0.99$. 
We begin by providing details about $\lambda_1(\xi)$. We have 
\begin{equation}
	\begin{split}
		\lambda_1(\xi(\alpha))&:= -\frac{\int_0^X (v''(x)+ v''''(x))\bar  v'(x)dx}{\int_0^X  v(x) \bar v'(x)dx}
		=:\frac{f(\alpha)}{g(\alpha)},
	\end{split}
	\label{innerprod}
\end{equation}
where
\begin{equation}
	\begin{split}
		v(x):= \sigma^2(x+i\omega'+\alpha)\sigma^{-2}(x+i\omega')\sigma^2(\alpha)e^{-2(x+i\omega')\zeta(\alpha)},
	\end{split}
	\label{vdef}
\end{equation}
and $\sigma(z)$ and $\zeta(z)$ are respectively the Weierstrass sigma and zeta functions with real half period $\omega$ and purely imaginary half-period  $i\omega'$. Here $v(x)$ is an eigenfunction, derived in \cite{KSF}, of the linearized gKS operator. Further, we have
\begin{equation}
	\begin{split}
		\omega = \frac{\pi}{\kappa},\quad \omega' = \frac{K(\sqrt{1-k^2})\pi}{K(k)\kappa}, \quad \xi = 2i\left(\zeta(\alpha)-\frac{\alpha}{\omega}\zeta(\omega)\right),\quad \lambda_{KdV}(\xi(\alpha))=-4\wp'(\alpha),
		\end{split}
	\label{some_defs}
\end{equation}
where $K=K(k)$ is the complete elliptic integral of the first kind, and $X = 2\omega$ is the period of the traveling wave. The parameter $\alpha \in \omega\Z\times i\R$ determines the Floquet parameter, $\xi(\alpha)$, and $\kappa=\mathcal{G}(k)$ satisfies,
\begin{equation}\label{def:kappa}
\displaystyle
\left(\frac{K(k)\mathcal{G}(k)}{\pi}\right)^2=
   \frac{7}{20}\frac{2(k^4-k^2+1)E(k)-(1-k^2)(2-k^2)K(k)}{(-2+3k^2+3k^4-2k^6)E(k)+(k^6+k^4-4k^2+2)K(k)},
\end{equation}
where $E=E(k)$ is the complete Elliptic integral of the second kind.


\subsection{Simplicity of KdV eigenvalues, (A1)}\label{simplicity} In this section we show that for $k = 0.99$, the nonzero KdV eigenvalues of the linearized KdV operator are simple for all Floquet parameters $\xi \in[0,2\pi/X)$ and $\lambda_{KdV} = 0$ only if $\xi = 0\mod \frac{2\pi }{X}$.
The KdV spectra are given by
$
\lambda_{KdV} = -4\wp'(\alpha),
$ where $\wp'$ is the derivative of the  Weierstrass elliptic function $\wp$, and $\alpha = n\omega + i\psi \omega'$. 
%
%
%
%
\subsubsection{Parametrization}
In order for the Floquet parameter, $\xi$, given in \eqref{some_defs}  to be real $(\Re(\alpha) = n\omega)$ $n\in \Z$; see \cite{S}.  From the quasi-periodicity of the Weierstrass elliptic functions,  \eqref{vdef} is invariant under the transformation $\alpha \to \alpha + 2n\omega + 2m\omega'$, $n,m\in \Z$, so that we may limit our study to $\alpha = \tilde n\omega + i\psi \omega'$ for $\tilde n = \{0,1\}$ and $\psi \in [-1,1]$.  
By the quasi-periodicity of $\sigma(z)$ with respect to $i\omega'$ and the mirror symmetry property $\sigma(\bar z) = \overline{\sigma(z)}$ and $\zeta(\bar z) = \overline{\zeta(z)}$, the transformation of \eqref{vdef} by  $\beta\to -\beta$ is equivalent to $v(z) \to c\overline{v(z)}$ for some non-zero constant $c\in \C$. Thus, when evaluating the stability condition (S1), it suffices to consider $\alpha = \tilde n\omega + i\psi \omega'$ for $\tilde n = \{0,1\}$ and $\psi \in [0,1]$.

\begin{lemma}
\label{lemma:xi_zero}
The KdV eigenvalue $\lambda_{KdV}(\xi) = 0$ only if $\xi = 0\mod\frac{2\pi}{X}$.
\end{lemma}

\begin{proof}
The zeros of $\wp'(\alpha)$, hence of $\lambda_{KdV}$, for $\alpha \in \{0,1\}\times i (-\omega',\omega']$ are precisely $\alpha = \omega,\ i\omega',$ and $ \omega + i\omega'$. Recalling the definition of $\xi$ given in \eqref{some_defs}, we see $\xi(\omega) = 0$.
Using the Weierstrass zeta addition property (see Section \ref{properties}), we find $ \omega \xi(i\omega') = 2i\left(\zeta(i\omega')\omega -i\omega'\zeta(\omega)\right) = 2i\left(\frac{-\pi i}{2}\right)=\frac{2\pi}{x}$.
Using the Weierstrass identities described in Section \eqref{properties}, we find
$
\xi(\omega+i\omega')-\xi(i\omega')= 2i(\zeta(\omega+i\omega') -\zeta(i\omega') -\zeta(\omega))
= i\left( \frac{\wp'(\omega)-\wp'(i\omega')}{\wp(\omega)-\wp(i\omega')}\right)
= 0,$
since $\wp'(\omega) = \wp'(i\omega') = 0$. Thus $\xi(\omega+i\omega') = \frac{2\pi}{X}$.
\end{proof}

\noindent Next we show that the nonzero KdV eigenvalues of the linearized KdV operator are simple for all Floquet parameters $\xi \in[0,2\pi/X)$. We begin by showing some important characteristics of $\lambda_{KdV}(\alpha)$ and $\xi(\alpha)$ which are demonstrated in Figure ~\ref{fig:simplicity}. 

\begin{center}
\begin{figure}[htbp]
	$
	\begin{array}{lr}
	   (a) \includegraphics[scale=.3]{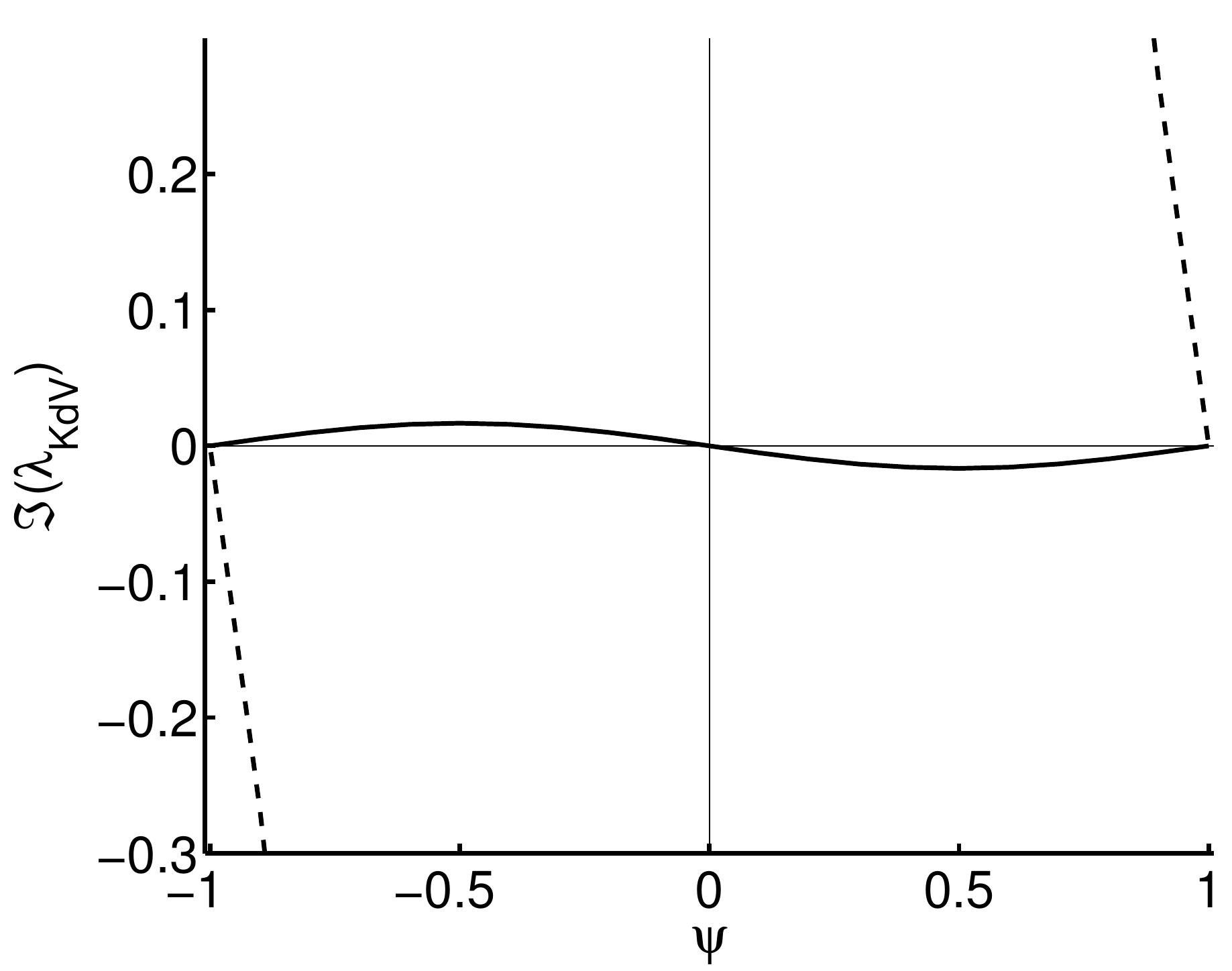} & (b)  	\includegraphics[scale=.3]{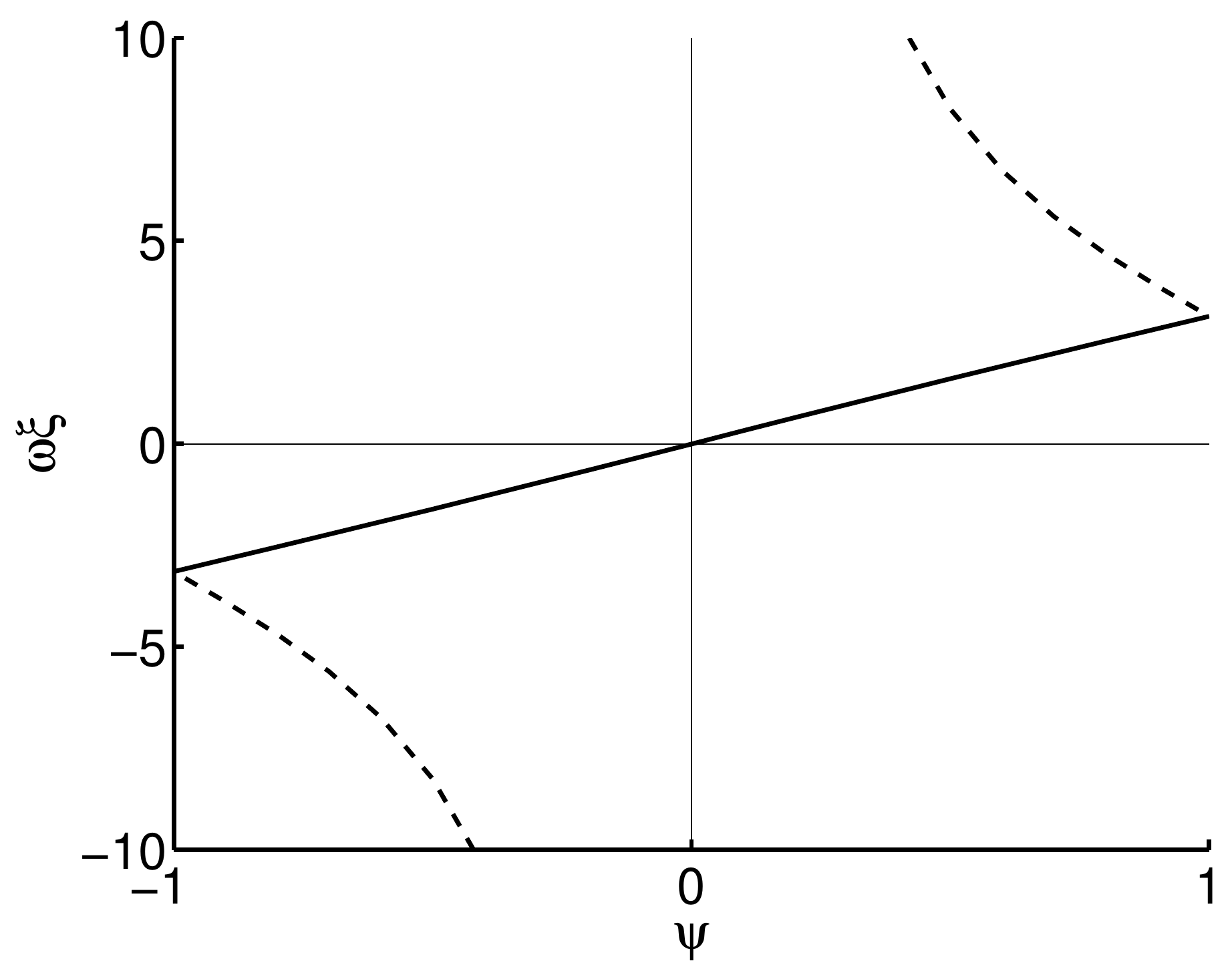}
	\end{array}
	$
    \caption{Solid lines correspond to $\alpha = \omega +i\psi \omega'$ and dashed lines to $\alpha = i\psi\omega'$. Here $k = 0.99$ (a) Plot of $\Im(\lambda_{KdV}(\alpha))$ against $\psi$. (b) Plot of $\xi(\alpha)$ against $\psi$.
    } \label{fig:simplicity}
\end{figure}
\end{center}

\begin{lemma}
\label{lemma:properties}
The following hold:
(1) $\lambda_0(\tilde n\omega-i\beta)= -\lambda_0(\tilde n\omega+i\beta)$, \ $\beta \in (0,\omega']$,
(2) $\pd{}{\beta}\Im\lambda_0(\omega+i\beta)|_{\beta=0} < 0$,
(3) $\pd{}{\beta}\Im\lambda_0(i\beta)<0$,\ $\beta \in (0,\omega']$,
(4) $\xi(\tilde n\omega-i\beta) = -\xi(\tilde n\omega + i\beta)$ for $\tilde n \in\{0,1\}$ and $\beta\in(0,\omega']$,
(5) $\pd{}{\beta}\xi(i\beta)<0$ for $\beta\in(0,\omega']$, 
(6) $\pd{}{\beta}\xi(\omega+i\beta)>0$ for $\beta\in[0,\omega']$,
\end{lemma}

\begin{proof}
To prove assertion (1), we recall that $\xi(\alpha) = 2i(\zeta(\alpha)-\alpha\zeta(\omega)/\omega)$, where $\alpha = n\omega + i\beta$. Since $\zeta(z) = -\zeta(-z)$, $\xi(n\omega -i\beta') = 2i(\zeta(n\omega-i\beta')-(n\omega-i\beta')\zeta(\omega)/\omega)=
2i(-\zeta(-n\omega+i\beta')-(n\omega-i\beta')\zeta(\omega)/\omega).
$
Using the quasi-periodicity of the Weierstrass zeta function yields $\xi(n\omega -i\beta') = 2i(-(\zeta(n\omega+i\beta')-2n\zeta(\omega))-(\tilde n\omega-i\beta')\zeta(\omega)/\omega)= -\zeta(n\omega+i\beta')+(\tilde n\omega+i\beta')\zeta(\omega)/\omega
=-\xi(n\omega+i\beta')$.

To establish assertion (2), we note that $\lambda_0(\omega+i\beta) = -4i\wp'(\omega+i\beta)$ so that $\pd{}{\beta}\lambda_0(\omega+i\beta) = -4i(6\wp^2(\omega+i\beta)-g_2/2)$. Then $\pd{}{\beta}\lambda_0(\omega+i\beta)|_{\beta=0} = -4i(6\wp^2(\omega)-g_2/2)= -4i(6(3g_3/g_2)^2-g_2/2)=(-2i/g_2^2)(27g_3^2-g_2^3)= (-2i/g_2^2)(16(e_1-e_2)^2(e_3-e_1)^2(e_3-e_2)^2)$ where $e_i = \wp(\omega_i)$. Thus $\pd{}{\beta}\Im(\lambda_0(\omega+i\beta))|_{\beta=0}<0$.

We establish assertion (3) by using the q-series representation of $\wp'(z+i\omega')$ (see Section \eqref{qseries}) whereby we have
$
\lambda_0(i\beta)  = - 4\wp'(i\beta)
= -4\wp'(i(\beta-\omega')+i\omega')
= -4\left(\frac{2\pi^3}{\omega^3}\right)\sum_{k=1}^{\infty} \frac{k^2q^k}{1-q^{2k}} \sin\left(\frac{ik\pi(\beta-\omega')}{\omega}\right).
$
Then$\frac{\partial}{\partial \beta} \lambda_0(i\beta) =$ $ -8i(\frac{\pi}{\omega})^4\sum_{k=1}^{\infty} \frac{k^3q^k}{1-q^{2k}}\cos\left(\frac{ik\pi(\beta-\omega')}{\omega}\right).$
Note that $\cos\left(\frac{ik\pi(\beta-\omega')}{\omega}\right) \geq 0$ for each $k$ and so $\Im(\frac{\partial}{\partial \beta}\lambda_0(i\beta)) < 0$. 

To establish assertion (4), we recall that $\xi(\alpha) = 2i\left(\zeta(\alpha)-\frac{\alpha}{\omega}\zeta(\omega)\right)$, $\alpha = n\omega + i\beta$. Then, noting that $\zeta(-z) = -\zeta(z)$ and using quasi-periodicity of the Weierstrass zeta function (see Section \eqref{properties}), we find
\begin{equation}
\begin{split}
\xi(n\omega -i\beta) &= 2i\left(\zeta(n\omega-i\beta)-\frac{n\omega-i\beta}{\omega}\zeta(\omega)\right)\\
&= 2i\left(-\zeta(-n\omega+i\beta)-\frac{n\omega-i\beta}{\omega}\zeta(\omega)\right)\\
&= 2i\left(-\left(\zeta(n\omega+i\beta)-2n\zeta(\omega)\right)-\frac{n\omega-i\beta}{\omega}\zeta(\omega)\right)\\
&  = -\zeta(n\omega+i\beta)+\frac{n\omega+i\beta}{\omega}\zeta(\omega) =-\xi(n\omega+i\beta).
\notag
\end{split}
\end{equation}

To prove assertion (5), 
we note that the Weierstrass zeta function in its q-series form is given by
$
\zeta(z) = \frac{\zeta(\omega)z}{\omega} + \frac{\pi}{2\omega}\cot(\pi z/(2\omega))+\frac{2\pi}{\omega}\sum_{k=1}^{\infty} \frac{q^{2k}}{1-q^{2k}}\sin(k\pi z/\omega),
$
where $q = e^{-\pi \omega'/\omega}$. Substituting this into the definition of the Floquet parameter yields
\begin{equation}
\begin{split}
\xi(i\beta) &= 2i(\zeta(i\beta)-\frac{i\beta}{\omega}\zeta(\omega))\\
&= 2i\left(\frac{\pi}{2\omega}\cot(i\pi\beta/2\omega)+\frac{2\pi}{\omega}\sum_{k=1}^{\infty} \frac{q^{2k}}{1-q^{2k}}\sin (k\pi i \beta/\omega)\right).
\end{split}
\end{equation}
Then
$
\pd{}{\beta}\xi(i\beta) = \frac{2\pi^2}{\omega^2}\left( \frac{1}{4\sin^2\left(\frac{i\pi\beta}{2\omega}\right)}-2\sum_{k=1}^{\infty} \frac{q^{2k}k}{1-q^{2k}}\cos\left(\frac{i k\pi \beta}{\omega}\right)\right).
$
Noting that $\sin^2(i\pi\beta/2\omega)<0$ and $\cos^2(k\pi i\beta/\omega)\geq 0$ for $\beta \neq 0$, $\beta \in \R$, we see that $\pd{}{\beta}\xi(i\beta)<0$.

We establish assertion (6) numerically. Using interval arithmetic to evaluate $\pd{}{\beta}\omega \xi(\omega+i\beta)$ on 100 evenly spaced subintervals of $[0,\omega']$, we find $\pd{}{\beta}\omega \xi(\omega+i\beta)\geq 2.983065564768185$ for $\beta \in[0,\omega']$.
\end{proof}
We now define $h(x,y):= (ci\lambda_0(\omega-ix\omega')-ci\lambda_0(iy\omega'))^2+(\omega\xi(\omega-ix\omega')-\omega\xi(iy\omega')-2\pi)^2$ where $c = -8\omega^3/(\pi\vartheta_1'(0))^3 $ for use in the following lemma.  

\begin{lemma}
\label{lemma:num_properties}
The following hold for $\psi_0 = 0.98$: (1) $h(1,1)=0$, (2)  $|\lambda_0(\omega+i\psi\omega')|<\lambda_0(i\psi_o\omega')$ for $\psi \in [-1,0]$, (3) $\xi(i\psi_0\omega')<3\pi/X$, (4) $h(x,y)$ is strictly convex for $(x,y)\in [0,1]\times [\psi_0,1]$, (5) $h(x,y)\neq 0$ for $(x,y)\in [0,1]\times[\psi_0,1]$ except when $(x,y) = (1,1)$.  
\end{lemma}

\begin{proof}
By Lemma \ref{lemma:properties} we have $h(1,1)= 0$. Properties (2)-(4) were proven numerically using interval arithmetic. (2) We find that $\lambda_0(i\psi_0\omega')\geq 0.05099537458926759$ and that $|\lambda_0(\omega+i\psi\omega')| \leq 0.01733062797697513$ for $\psi \in [0,1]$. (3) We find $\omega\xi(i\psi_0\omega') \leq 3.289285089296752<3\pi\omega/X = 3\pi/2$. (4) We find $h_{xx}\geq 17.30449947025409 $ and $\Delta = h_{xx}h_{hh} - h^2_{xy} \geq 10504.76128061587$ for all $(x,y) \in [0,1]\times [\psi_0,1]$. (5) By (1) and (4) just shown.  
\end{proof}

With these general properties in place, we are ready to show simplicity.
\begin{proposition}\label{p:simplicity}
For $k = 0.99$, the nonzero KdV eigenvalues are simple.
\end{proposition}

\begin{proof}
The result of the theorem follows from the lemmata  \ref{lemma:properties} and \ref{lemma:num_properties} together with the fact that $\xi(\omega)=0$ and $\xi(\omega+i\omega')= \xi(i\omega')=2\pi/X$ shown in the proof of Lemma \ref{lemma:xi_zero}. 
\end{proof}

\subsection{Distinctness of $\alpha_j$, (A2)}\label{distinct}

Condition $A_2$ of \cite{JNRZ} is that the $\{\alpha_j\}$, $j = 1,2,3$, be distinct where 
\begin{equation}
\lambda_{KdV,j}(\xi) = i \alpha_j(\xi)\xi = i\alpha_j^0 \xi + o(|\xi|^2),
\notag
\end{equation}
and the $\lambda_{KdV,j}(\xi)$ are the three critical modes of the KdV linearized (Bloch) operator about a periodic wave. In \cite{JNRZ} it is noted that the $\alpha_j$ are precisely the eigenvalues of the Whitham modulational equations or the characteristic velocities which can be expressed in terms of the Riemann invariants $w_1\leq w_2\leq w_3$, $k^2 = \frac{w_2-w_1}{w_3-w_1}$ as 
\begin{equation}
\begin{split}
V_i(w_1,w_2,w_3) &= \frac{w_1+w_2+w_3}{3}+\frac{2(w_3-w_2)}{3}b_i,\\
b_1= \frac{k^2K(k)}{E(k)-K(k)},\quad b_2 &= \frac{k^2(1-k^2)K(k)}{(1-k^2)K(k)-E(k)},\quad b_3 = \frac{(1-k^2)K(k)}{E(k)}.
\end{split}\label{eq:distinct}
\end{equation}

Since $k=0.99$, we have from $w_1\leq w_2\leq w_3$ and $k^2 = \frac{w_2-w_1}{w_3-w_1}$ that  $w_3 > w_1$ and hence the $V_i$ are distinct iff the $b_i$ are distinct. Using interval arithmetic, we find for $k = 0.99$ that
$
b_1 \in [  -1.41307042217012,  -1.41307042217007], \
b_2\in [  -0.06807580473435,  -0.06807580473434],$ and $
b_3 \in [   0.06494693392736,   0.06494693392737],
$
so that the $\alpha_j$ are distinct.

\subsection{Stability condition (S1)}\label{showS1}

In this section we show that the stability condition (S1), given by equation \eqref{lambda1}, holds for $k = 0.99$. We recall the definition of $\Re(\lambda(\xi))$, $f(\alpha)$, and $g(\alpha)$ given in equation \eqref{innerprod}. In Section \ref{gspecial}, we prove that $g(\alpha) < 0 $ for $\alpha = i\beta$, $\beta\in[0,\omega']$. In Section \ref{zeros}, we show that $f(\alpha) = g(\alpha)=0$ for $\alpha \in \{\omega,\omega+i\omega',i\omega'\}$. We are not able to evaluate $f(\alpha)$ and $g(\alpha)$ explicitly for other values of $\alpha$, so we use interval arithmetic for these values. In Section \ref{chebyshev} we describe the analytic interpolation we use to compute the stability condition (S1). Finally, in Section \ref{formulation} we reformulate the stability condition $\Re(\lambda_1(\xi))$ for convenience and provide computational details. 


\subsubsection{Case $\alpha = i\beta$}\label{gspecial}

In the case that $\alpha = i\beta$, it can be shown that the denominator of \eqref{innerprod}, $g(\alpha) = \int_0^X \bar v(x) v'(x)dx$, is positive for all $k\in(0,1)$ and $\beta \in (0,\omega')$.

%
%

\bl
The integral $\int_0^X \bar v(x) v'(x)dx > 0$ for all $k\in (0,1)$ for $\alpha = i\beta$ with $\beta \in (0,\omega')$. 
\el
\begin{proof}
Define
\begin{equation}\label{wdef}
w(x):= \sigma^2(x+i\omega'+\alpha)\sigma^{-2}(x+i\omega'),\quad c(\alpha):=e^{-2i\omega'\zeta(\alpha)}\sigma^{-2}(\alpha).
\end{equation}
Then $v(x) = c(\alpha) w(x) e^{-2x\zeta(\alpha)}$. Now 
$
w'(x) = 2w(x)(\zeta(x+i\omega'+\alpha)-\zeta(x+i\omega'))
$
 and so we have
\begin{equation}
\begin{split}
v'(x)&= 2c(\alpha)\left[\zeta(x+i\omega'+\alpha)-\zeta(x+i\omega')-\zeta(\alpha)\right]w(x)e^{-2x\zeta(\alpha)},\ \mathrm{and} \\
\overline{v(x)}v'(x)&=  2\left|c(\alpha)w(x)e^{-2x\zeta(\alpha)}\right|^2\left(\zeta(x+i\omega'+\alpha)-\zeta(x+i\omega')-\zeta(\alpha)\right).
\end{split}
\end{equation}
By a well-known property \eqref{prop:addition},
$
\zeta(x+i\omega'+\alpha)-\zeta(x+i\omega')-\zeta(\alpha)= \frac{1}{2}\frac{\wp'(x+i\omega')-\wp'(\alpha)}{\wp(x+i\omega')-\wp(\alpha)}.
$
From the q-series representation of $\wp$ \eqref{qseries}, we see that $\wp(x+i\omega')-\wp(i\beta)\in \R$, and $\wp'(x+i\omega)\in \R$, and $\Re(\wp'(i\beta)) = 0$. Thus, 
\begin{equation}\label{sign1}
\textrm{sign}(\Im[\zeta(x+i\omega'+\alpha)-\zeta(x+i\omega')-\zeta(\alpha)])= -\textrm{sign}(\Im(p'(i\beta)))\textrm{sign}(\wp(x+i\omega')-\wp(\alpha)).
\end{equation}
Now 
\begin{equation}
p'(i\beta)= -\frac{i\pi^3}{\omega^3}\left( \frac{e^{\pi \beta/2\omega}+e^{-\pi\beta/2\omega}}{(e^{\pi \beta/2\omega}-e^{-\pi\beta/2\omega})^3} - 2\sum_{k=1}^{\infty} \frac{k^2q^{2k}}{1-q^{2k}}(e^{\pi\beta/2\omega}-e^{-\pi\beta/2\omega})\right),
\end{equation}
which, upon inspection, yields $\Im(\wp'(i\beta))< \wp'(i\omega')=0$ for $\beta \in (0,\omega')$. Note that 
$
\frac{\partial}{\partial \beta} (\wp(x+i\omega')-\wp(i\beta))= -i\wp'(i\beta)
= \Im(\wp'(i\beta))
< \wp'(i\omega')
=0,
$
so that, employing \eqref{qseries}, we have
$
\wp(x+i\omega')-\wp(i\beta)\geq \wp(x+i\omega')-\wp(i\omega')
= \frac{2\pi^2}{\omega^2}\sum_{k=1}^{\infty} \frac{kq^k}{1-q^{2k}}\left(1-\cos\left(\frac{k\pi x}{\omega}\right)\right)
\geq 0.
$
Then from \eqref{sign1} we have that $\Im[\zeta(x+i\omega'+\alpha)-\zeta(x+i\omega')-\zeta(\alpha)]>0$ so that $\Im(\bar v(x)v'(x))\geq 0$,  hence $\Im(\int_0^X\bar v(x)v'(x))dx>0$.

\end{proof}

\subsubsection{Zeros of $f(\alpha)$ and $g(\alpha)$}\label{zeros} The denominator and numerator of \eqref{innerprod} is zero when $\alpha = \omega,\ i\omega',$ or $ \omega+i\omega'$, which cancellation can not be resolved with interval arithmetic. Here we show the existence of these zeros analytically. 

\begin{lemma} \label{lemma:zeros}
The functions $f(\alpha)$ and $g(\alpha)$ defined in equation \eqref{innerprod} have zeros at $\alpha = \omega,\ i\omega'$, and $\omega+i\omega'$.
\end{lemma}

\begin{proof}
We use the quasi-periodicity and addition properties of $\sigma(\cdot)$ given in \eqref{prop:periodicity} and \eqref{prop:addition} to simplify $v(x)$ given by equation \eqref{vdef}. We obtain, fixing $(\tilde n,\psi)\in \{(1,0),(0,1),(1,1)\}$,
\begin{equation}
\begin{split}
v(x)&= \sigma^2(x+i\omega'+\tilde n\omega+i\psi \omega')\sigma^{-2}(x+i\omega')\sigma^{-2}(\tilde n\omega +i\psi\omega')e^{-2(x+i\omega')\zeta(\tilde n\omega + i\psi \omega')}\\
&= (-1)^{\tilde n + \psi}\frac{ \sigma(x+i\omega' +\tilde n \omega + i\psi \omega')\sigma(x+i\omega'-\tilde n \omega -i\psi \omega')}{\sigma^{2}(x+i\omega')\sigma^{2}(\tilde n \omega + i\psi \omega')}e^{-2(x+i\omega')(\tilde n \zeta(\omega)+\psi \zeta(i\omega')+ \zeta(\tilde n \omega + i\psi \omega'))}\\
&= (-1)^{\tilde n + \psi + 1}(p(x+i\omega')-p(\tilde n \omega + i\psi \omega')).
\end{split}
\end{equation}
Note that in the case $\tilde n = \psi = 1$, $\tilde n \zeta(\omega)+\psi \zeta(i\omega')+ \zeta(\tilde n \omega + i\psi \omega')) =0$ since $\zeta(\tilde n \omega + i\psi \omega') = -\eta_2$ and $\eta_1+\eta_2+\eta_3 = 0$. Since \eqref{vdef} is invariant by multiplication of $v(\cdot)$ by a non-zero constant, we may take $v(x) = p(x+i\omega')-p(\tilde n \omega + i\psi \omega')$.
The derivatives of $v(x)$ are given by
$
v'(x) = \wp'(x+i\omega')$, 
$v''(x)= \wp^2(x+i\omega')-\frac{g_2}{2}$, and
$v'''(x) = 12 \wp(x+i\omega')\wp'(x+i\omega')$,
where $g_2$ and $g_3$ are the Weierstrass elliptic function invariants. The mirror symmetry and periodicity properties of the derivative of the Weierstrass elliptic function lead us to conclude $\bar v'(x) = v'(x)$ for $\alpha = i\omega'$. Recalling that $X = 2\omega$, we see that
\begin{equation}
g(\tilde n\omega + i\psi \omega') = \int_0^{2\omega} (\wp(x+i\omega')-\wp(\tilde n + i\psi\omega'))\wp'(x+i\omega')dx
= \left. \frac{1}{2}\wp^2(x+i\omega')-\wp(\tilde n + i\psi\omega')\wp(x+i\omega')\right|_0^{2\omega}
= 0
\label{feval}
\end{equation}
since $\wp$ is $2\omega$ periodic. Similarly,
\begin{equation}
\begin{split}
f(\tilde n \omega + i\psi \omega') 
&= \int_0^{2\omega}v''(x)v'(x)-v'''(x)v''(x)dx\\
&= \int_0^{2\omega} (6\wp^2(x+i\omega')-g_2/2)(1-12\wp(x+i\omega'))\wp'(x+i\omega')dx
= 0.
\end{split}
\label{geval}
\end{equation}
\end{proof}

%
%

\subsubsection{Problem formulation}\label{formulation}

Evaluating \ref{innerprod} is poorly conditioned numerically when $\alpha \approx 0$ since $\sigma(0) = 0$ and $\zeta(0) = \infty$. Note however that \eqref{innerprod} is invariant under multiplication of \eqref{vdef} by a nonzero constant, $c(\alpha)$. Letting $c(\alpha)= e^{-2\omega'\zeta(\alpha)i}/\sigma(\alpha)^2$ yields
$
v(x):= \frac{\sigma^2(x+i\omega' + \alpha)}{\sigma^2(x+i\omega')}e^{ -2\zeta(\alpha)},
$
which by \eqref{w_of_theta} may be written as
\begin{equation}
v(x)= \vt_1^2(\pi z_1/2\omega)\vt_2^{-2}(\pi z_2/2\omega)e^{\eta_1z_1^2/\omega}e^{-\eta_1z_2^2/\omega}e^{\gamma x}
= c(\alpha) \vt_1^2(\pi z_1/2\omega)\vt_2^{-2}(\pi z_2/2\omega)e^{i\xi x},
\label{newvdef}
\end{equation}
where $\vartheta_1(\cdot)$ is the Jacobi Theta function, $z_1 = x+i\omega' + \alpha$, $z_2 = x+i\omega'$, and $\xi\in \R$ is given by \eqref{some_defs}. Once again, \eqref{innerprod} permits we drop the constant $c(\alpha)$, so we may take
$
v(x)= w(x)e^{\gamma x},
$
where
$
w(x):= \vt_1^2(\pi z_1/2\omega)\vt_2^{-2}(\pi z_2/2\omega),
$
and $\gamma = i\xi$.
There is still a singularity since $|\xi(\alpha)| \to \infty$ as $\alpha \to 0$, see \eqref{some_defs}.  However, the singularity is no longer coupled to the spatial variable $x$ making it possible to factor out the singular parts from the integrals given in \eqref{innerprod2}.  
We make a few simple changes that make it possible to compute the Jacobi Theta function series with fewer terms resulting in smaller interval error bounds. We begin by recalling that $v(x)$ is $X = 2\omega$ periodic allowing us to center the integrals about $x = 0$. We also make a change of variables $x\to \omega x$, and reduce the number of derivatives of $v(x)$ we must compute via integration by parts yielding,
%
%
\begin{equation}
\lambda_1 
= \frac{\int_{-1}^{1} \left[\frac{\partial}{\partial x}v(\omega x)+\frac{1}{\omega^2}\frac{\partial^3}{\partial x^3}v(\omega x)\right]\frac{\partial^2}{\partial x^2}\bar v(\omega x)dx}{\omega^2\int_{-1}^{1} v(\omega x)\frac{\partial}{\partial x}\bar v(\omega x)dx}=:\frac{\tilde f(\alpha)}{\tilde g(\alpha)}=:\frac{\tilde f_1(\alpha)+\tilde f_2(\alpha)/\omega^2}{\tilde g(\alpha)}.
\label{innerprod2}
\end{equation}

Now
\begin{equation}
\begin{split}
v(x) = w(x)e^{\gamma x},\quad
v'(x)& = w'(x) e^{\gamma x} + \gamma w(x)e^{\gamma x},\quad
v''(x)  = w''(x) e^{\gamma x} + 2\gamma w'(x) e^{\gamma x} + \gamma^2 w(x) e^{\gamma x},\\
v'''(x) &= w'''(x) e^{\gamma x} +3\gamma w''(x) e^{\gamma x} + 3\gamma^2 w'(x) e^{\gamma x} + \gamma^3 w(x) e^{\gamma x}.
\end{split}
\label{vandder}
\end{equation}
Hence
\begin{equation}
\begin{split}	
v(\omega x) \bar v'(\omega x) &= w(\omega x) (\bar w'(\omega x) + (i\xi) \bar w(\omega x)),\quad
v(\omega x) \bar v''(\omega x) = \sum_{n=0}^{3} \gamma^n c_n(x),\\
v'''(\omega x) \bar v''(\omega x)& = \sum_{n=0}^{5} \gamma^n h_n(x).
\end{split}
\label{factored}
\end{equation}

%
%
%


The advantage of the formulation \eqref{factored} is that $c_n(z)$ and $h_n(z)$ are analytic in $z$ and so we may easily compute $\int_{-1}^1 c_n(x)dx$ and $\int_{-1}^1 h_n(x)dx$ with interval arithmetic as called for in \eqref{innerprod2}. This allows us to compute $f(\alpha)$ and $g(\alpha)$ in the limit $\alpha\to 0$, where $\gamma(\alpha) = i\xi(\alpha) \to \infty$.

Because of the mirror symmetry of $\vartheta_1(x)$, the conjugate of $w(x)$ and its derivatives are analytic functions in the variables $x$ and $\beta$. Then the integrands in the numerator and denominator of \eqref{innerprod2} are analytic in both the variables $x$ and $\beta$ on an open, connected set not containing zeros of $\vt_1(\pi(x\pm i\omega')/2)$ or the poles of the Weierstrass Zeta function. Hence, we may use analytic interpolation.

For convenience we will set $\beta = \psi \omega'$ and parametrize by $\psi\in[0,1]$ instead of $\beta$.


Interpolation with our bootstrapping method works very well, but evaluating the interpolating polynomial with interval arithmetic still requires some care. Suppose we are interested in verifying $f(x):[-1,1]\to \R$ has no zeros on an interval $[a,b]\subset [-1,1]$. Let $p_N(x)$ be the Chebyshev interpolant of $f(x)$ on $[-1,1]$ with N nodes with interpolation error $\eps_N$. We seek to verify that, without loss of generality, $p(x) > \eps_x$ for $x\in[a,b]$ which implies that $f(x) > 0$ for $x\in [a,b]$. Making the transformation, $x= \cos(\theta)$, it is equivalent to verify $p(\theta) > \eps_N$ for $\theta \in [\cos^{-1}(b),\cos^{-1}(a)]\subset [0,\pi]$. Note that $p(\theta) = \sum_{n=0}^N c_n \cos(n\theta)$. Now we Taylor expand $p(\theta) = \sum_{n=0}^5 p^{(n)}(\theta_0)(\theta-\theta_0)^n + p^{(6)}(\tilde \theta)(\theta-\theta_0)^6$ for some $\tilde \theta \in [\theta_0,\theta]$, as described in Section \ref{evaluate}.


\subsubsection{Interpolation results}

In this section we provide details of our numerical verification of the stability condition (S1),
that is we show that $\lambda_1(\xi)<0$, where $\lambda_1$ is as described in equation \eqref{innerprod2}. In all of our computations, we use Version 6 (30 March 2010) of the MATLAB based interval arithmetic package INTLAB \cite{R}. We evaluate the Jacobi Theta function $\vartheta_1(\cdot)$ by using its q-series representation given in equation \eqref{qseries}. We also evaluate $\xi(\alpha)$ by using the q-series representation of the Weierstrass zeta function and simplifying terms where possible. Explicit error bounds and Matlab code are provided in \cite{Bdoc}. 

Our general strategy is to prove numerically using interval arithmetic that the functions $\tilde f(\alpha)$ and $\tilde g(\alpha)$ defined in \eqref{innerprod2}	characteristically are as depicted in Figure \ref{fig96}, hence $\lambda_1(\xi) < 0$. The following lemma will aid us in our proof by allowing us to compute $1/\omega(k)^2$, used in evaluating $\tilde f$, at the left and right endpoints of an interval in $k$ in order to obtain a tighter interval enclosure of $1/\omega^2$.

\begin{lemma}\label{lemma:kappa-monotone}
For $k\in [0.9,9999999]$, $\pd{}{k}\kappa^2(k) \leq -2.248783289537847.$
\end{lemma}

\begin{proof}
We solve for $\kappa^2(k)$ in \eqref{def:kappa}, compute its derivative, and evaluate the derivative with interval arithmetic to determine that $\pd{}{k}\kappa^2(k) \leq -2.248783289537847$ for $k \in[ 0.9,0.9999999]$. See section kappa\_lemma in \cite{Bdoc} for details.
\end{proof}

\begin{corollary}
For $k\in [0.9,9999999]$, $\frac{1}{\omega^2(k)} = \frac{\kappa(k)^2}{\pi^2}$ is monotone decreasing.
\end{corollary}

Now we present the main lemma of this section.

\begin{lemma}\label{p:stab}
For $k = 0.99$, $\lambda_1(\xi(\alpha)) < 0$ for $\alpha = \tilde n +i\psi \omega'$ where $\tilde n \in \{0,1\}$ and $\psi \in [0,1]$. 
\end{lemma}  

\begin{proof}
Let $k = 0.99$ and consider the case $\alpha = \omega + i\psi \omega'$, $\psi \in [0,1]$. To evaluate the sign of $\lambda_1(\xi(\alpha))$ defined in equation \eqref{innerprod2}, we determine the signs of $\tilde f(\alpha) = \tilde f_1(\alpha)+\tilde f_2(\alpha)/\omega^2$ and $\tilde g(\alpha)$ and their derivatives with respect to $\psi$.  See \eqref{innerprod2} for the definitions of these functions. We use Chebyshev interpolation to approximate the integrands appearing in the definition of these functions. To apply the approximation error bound given in equation \eqref{interp-bound2}, we must specify a radius $\rho_x > 1$ so that the integrands to be interpolated are analytic on and inside the stadium $E_{\rho_{x}} := \{(\rho_xe^{i\theta}+e^{-i\theta}/\rho_x)/2| \theta \in [0,2\pi]\}$. The functions of interest are analytic so long as $\rho_x$ is chosen sufficiently small to avoid the poles of $v(x)$. Recall from equation \eqref{newvdef} and subsequent discussion that $v(x) := \vt_1^2(\pi (x+i\omega' +\tilde n \omega+ i\psi \omega')/2\omega)\vt_1^{-2}(\pi (x+i\omega')/2\omega)e^{i\xi x}$. The Jacobi Theta function $\vt_1(\cdot)$ is analytic and its zeros are the set $\{m\pi +n\pi i\omega'/\omega | m,n\in \N \}$. A simple computation shows that if we take $|\Im(x)|< \omega'/\omega$, then $\vt_1(\pi (x+i\omega')/2\omega) \neq 0$. Setting $c = 0.9\omega'/\omega$ and defining $\rho_x := c + \sqrt{c^2+1}\in [   1.50919391484325,   1.50919391484326] $, we have that $\vt_1(\cdot)$ has no zeros inside or on the stadium $E_{\rho_{x}}$, hence $v(x)$, and thus the integrands to be interpolated, are analytic in $x$ inside and on the stadium $E_{\rho_{x}}$. 

Next we bound the modulus of the relevant functions on $E_{\rho_{x}}$. To obtain this bound, we must find a lower bound on $\vt_1(\pi (x+i\omega')/2\omega)$ for $x\in E_{\rho_{x}}$, but by the parity and mirror symmetry properties of the Jacobi Theta function, we actually only need to consider $\vt_1(\pi(x+i\omega')/2\omega)$ for $x\in \{(\rho_xe^{i\theta}+e^{-i\theta}/\rho_x)/2| \theta \in [0,\pi/2]\}$. We used 8000 intervals of even width to step through the parameter $\theta$ to obtain the lower bound. For example, for the $\theta$ interval [   0.02513274122871,   0.02532909076957], we find $|\vt(\pi(x(\theta)+i\omega')/2\omega)|\geq M$ where $M \in
[   2.12547636195776,   2.12643781736408]$. Using INTLAB's inf function, we compute a machine-representable lower bound on the intervals representing the modulus of  $\vt(\pi(x(\theta)+i\omega')/2\omega)$ and then take the minimum of all of these. We find that 0.324078550629158 is a lower bound for $|\vt(\pi(x(\theta)+i\omega')/2\omega)|$ for $x\in \rho_{x}$. See the documentation for lower\_bound.m given in \cite{Bdoc} for the code and details. Next we use the q-series representation of $\vt_1(\cdot)$ to obtain an upper bound on $\vt_1(\cdot)$ and its first four derivatives analytically, which we then compute with interval arithmetic. As an example, in computing an upper bound $M_x$ for interpolation in $x$ of the integrand associated with $\tilde f(\alpha)$, we find $M_x\in 
[10^{16}\times  7.08053733846663,  10^{16}\times 7.08053733846667]$. We find that $M_x = 2.55e+23$ is an upper bound for all functions to be interpolated. See bound\_theta1\_m.m in \cite{Bdoc} for the source code and truncation error bounds. We use the lower bound and upper bounds we found for $\vt_1(\cdot)$ and its derivatives to bound the integrands we are to interpolate. See bound\_numer.m in \cite{Bdoc} for details. 

Substituting the upper bound $M_x$ we just found and $\rho_x$ into equation \eqref{interp-bound2}, we find that $N = 241$ interpolation nodes suffices to approximate the integrands of interest for any $\psi \in [0,1]$ with an interpolation error bound of $\eps \in 10^{-18} \times [   0.96880172896556,   0.96880172896646]$. This serves as an example of the efficiency of analytic interpolation, where a crude bound of order $10^{23}$ and 241 interpolation nodes result in an approximation error of order $10^{-18}$. See N\_nodes.m in \cite{Bdoc} for computational details.

Hereafter, when we compute $\tilde f_1(\cdot)$, $\tilde f_2(\cdot)$, $\tilde g(\cdot)$, or their derivatives with respect to $\psi$, we do so with interval arithmetic, Chebyshev interpolation in $x$ with error bounds, and Chebyshev integration described in Section \ref{chebyshev}. Our next step is to interpolate in the variable $\psi$. As we did for $x$, we choose $\rho_{\psi} > 1$. Note that $\rho_{\psi}$ must be chosen so that $\xi(\omega + i\psi \omega')$ does not have a pole inside or on the stadium $E_{\rho_{\psi}}$. The poles of $\xi$ are $z = 2m \omega + 2n \omega'$. Setting
$2m\omega + 2n\omega' = \omega + i\psi \omega'$ with $\psi = 1/2 + \tilde \psi/2$, we
find that $|\Im(\tilde \psi)|< -\frac{\pi}{\log(q)}$ is necessary and sufficient to ensure analyticity. We set $c:= 0.95 \pi /|\log(q)|$ and set $\rho_{\psi} = c+\sqrt{c^2+1}\in [   4.07266431471885,   4.07266431471886]$. We bound the modulus of $\tilde f_1(\cdot)$, $\tilde f_2(\cdot)$, $\tilde g(\cdot)$, and their derivatives with respect to $\psi$ in the same way we bound the integrands when interpolating in the variable $x$. Here we let $\psi\in E_{\rho_{\psi}}$, but the same lower bound on $|\vt(\pi(x(\theta)+i\omega')/2\omega)|$ applies, and we may bound the modulus of the integrands as before and then multiply that number by the width of the interval on which we integrate. We find that the relevant functions are bounded above in modulus by $ M_{\psi} \in  10^{19}\times [   7.31602374789772,   7.31602374789773] $. See numer\_bound.m in \cite{Bdoc} for details. Substituting the bound $M_{\psi}$ and $\rho_{\psi}$ into \eqref{interp-bound2}, we find that $N = 42$ interpolation nodes suffice to assure interpolation approximation error does not exceed $\eps_{\psi} \in   10^{-5} \times [   0.32664460793827,   0.32664460793829]$.

We used the above error bounds, number of interpolation nodes, and interpolation error bounds to obtain interval inclusions of $\tilde f(\omega+i\psi\omega')$ and $\tilde g(\omega+i\psi\omega')$ for $\psi \in [0,1]$. For $a = 0.1$ and $b = 0.9$, we verified that $\tilde f(\omega+i\psi\omega') \geq 5.353348309012314>  0$ and $\tilde g(\omega+i\psi\omega')\leq -0.840361567746836 < 0$ for $\psi \in [a,b]$, $\pd{}{\psi}\tilde f(\omega+i\psi\omega')\geq 11.811632963420863 > 0$ and $\pd{}{\psi}\tilde g(\omega+i\psi\omega')\leq -5.495247935051245 < 0$ for $\psi \in [0,a]$, and $\pd{}{\psi}\tilde f(\omega+i\psi\omega')\leq -5.278723771094692\times 10^2 < 0$ and $\pd{}{\psi}\tilde g(\omega+i\psi\omega')\geq 1.163460415865856\times 10^2 > 0$ for $\psi \in [b,1]$.

Recalling that $\tilde f(\omega+i\psi\omega')=g(\omega+i\psi\omega') = 0$ for $\psi \in\{0,1\}$, as shown in Section \ref{zeros}, we have numerically verified $\lambda_1(\xi(\alpha)) < 0$ for $\alpha = \omega+i\psi \omega'$. In Figure \ref{fig96} we plot $\tilde f(\alpha)$, $\pd{}{\psi}\tilde f(\alpha)$, $\tilde g(\alpha)$, and $\pd{}{\psi}\tilde g(\alpha)$ for $\alpha = \omega + i\psi \omega'$. 

Next, we treat the case $\alpha = i\psi \omega'$. Since $\xi(i\psi \omega')\to \infty$ as $\psi \to 0$, we must treat this case differently. Indeed, we cannot obtain interpolation bounds since $\tilde f(\cdot)$ includes powers of $\xi(i\psi \omega')$ in its definition. We avoid this problem by interpolating the coefficient functions $c_j(x)$ and $h_j(x)$ defined in equation \eqref{factored} and then forming $\tilde f(i\psi \omega')$ as a polynomial expansion in $\xi$. The general strategy is to use interval arithmetic to show that $\pd{}{\psi} \tilde f(i\psi\omega')< 0 $ for $\psi \in[b,1]$ for some $0<b<1$. Using polynomial root bounds, we show that $f(i\psi\omega') $ has no zeros for $\psi \in [0,a]$ for some $0<a<b$. Finally, we evaluate with interval arithmetic the polynomials given in \eqref{factored} using the interpolated coefficients to verify that $\tilde f(i\psi\omega') > 0$ for $\psi \in [a,b]$. Recalling that we have shown that $\tilde f(i\omega') = 0$ in Section \ref{zeros} and that $\tilde g(i\psi \omega')<0$ in \ref{gspecial}, this gives that $\lambda_1(\xi(\alpha)) < 0$ for $\alpha = i\psi \omega'$. In Figure \ref{fig96} we plot $\tilde f(i\psi \omega')$ and $\pd{}{\psi} \tilde f(i\psi \omega')$.

We use the same $\rho_x$ and $\rho_{\psi}$ as we used for $\alpha =\omega+i\psi \omega'$. We find that $M_x = 1.421928494683729 \times 10^{23}$ is an upper bound on the modulus of the coefficient functions $c_j(x)$ and $h_j(x)$ for $x\in E_{\rho_{x}}$, and $M_{\rho_{\psi}}=1.938556518787788\times 10^{19}$ is an upper bound for the integrals of the coefficient functions $c_j(x)$ and $h_j(x)$ as functions of $\psi$. Substituting the bounds into equation \eqref{interp-bound2}, we find $N_x = 240$ interpolation nodes suffices to guarantee the approximation error of the coefficient functions does not exceed $\eps_x\in 10^{-18}\times [0.81368237383580,   0.81368237383655]$. Similarly $N_{\psi}  = 41$ guarantees an error bound of $\eps_{\psi}\in 10^{-5}\times[   0.35249865941173,   0.35249865941175]$.
 
Formulating $\tilde f(\alpha)$ from the coefficients $c_j(\psi,x)$ and $h_j(\psi,x)$, we have $\tilde f(\alpha) = \sum_{k=0}^5 p_k(\psi,x)(i\xi)^k$. We verify for $0<\xi\leq 10^{-3}$, using interval arithmetic to evaluate the Chebyshev interpolation polynomials, that $|p_0| \leq 0.078295154350480$, $|p_1|\leq 0.047185115122197$, $|p_2|\leq  0.012816747371156$, $|p_3|\leq 2.007891703388547$, $|p_4|\leq 1.862679700692074e-04$, and $|p_5| \geq 0.062630009852173$. Using the general upper bound $R = 1 + (1/a_n)\max(a_0,a_1,...,a_{n-1})$ on roots of polynomials $p(x) = \sum_{k=0}^n a_k x^k$, we find that $\tilde f(i\psi \omega') = 0$ for $0<\psi\leq 10^{-3}$ implies that $|\xi|\leq 33.059578277694790$. We compute  $\xi(10^{-3})\in 10^{3}[   4.25236053848855,   4.25236053849688]$. From Lemma \ref{lemma:properties} part (3), $\xi$ is decreasing on the interval $[0,10^{-3}]$, hence $\tilde f(i\psi\omega')\neq 0$ for $\psi \in [0,10^{-3}]$. 

Next we evaluate the Chebyshev interpolation polynomial to determine where $\tilde f(i\psi \omega')> 0 $. We find that this holds on the interval $[a,b]$ where $a = 0.00109909909910 $ and $b =  0.99863463463463$. Finally we verify with direct interval arithmetic computation that $\pd{}{\psi}\tilde f(i\psi\omega')<0$ on the interval $[b,1]$. Thus, we have shown that $\tilde f(i\psi \omega') > 0$ for $\psi \in [0,1]$. Then by Lemma \ref{gspecial} and equation \eqref{innerprod2}, $\lambda_1(\alpha)< 0$ for $\alpha = i\psi \omega'$ with $\psi \in [0,1]$. This completes the proof of the lemma.
\end{proof}

\begin{figure}[htbp]
 \begin{center}
$
\begin{array}{lcr}
(a) \includegraphics[scale=0.2]{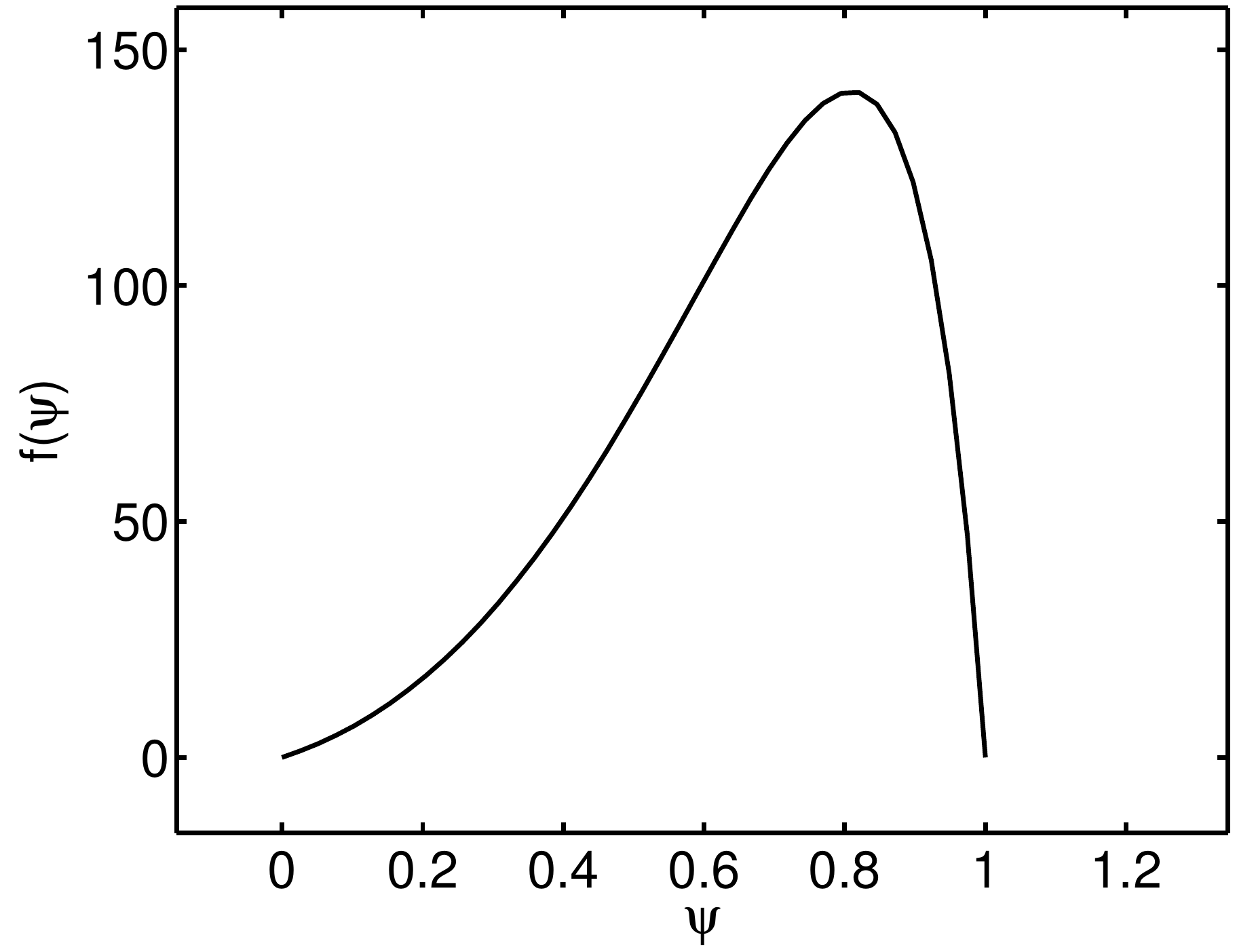} &
(b) \includegraphics[scale=0.2]{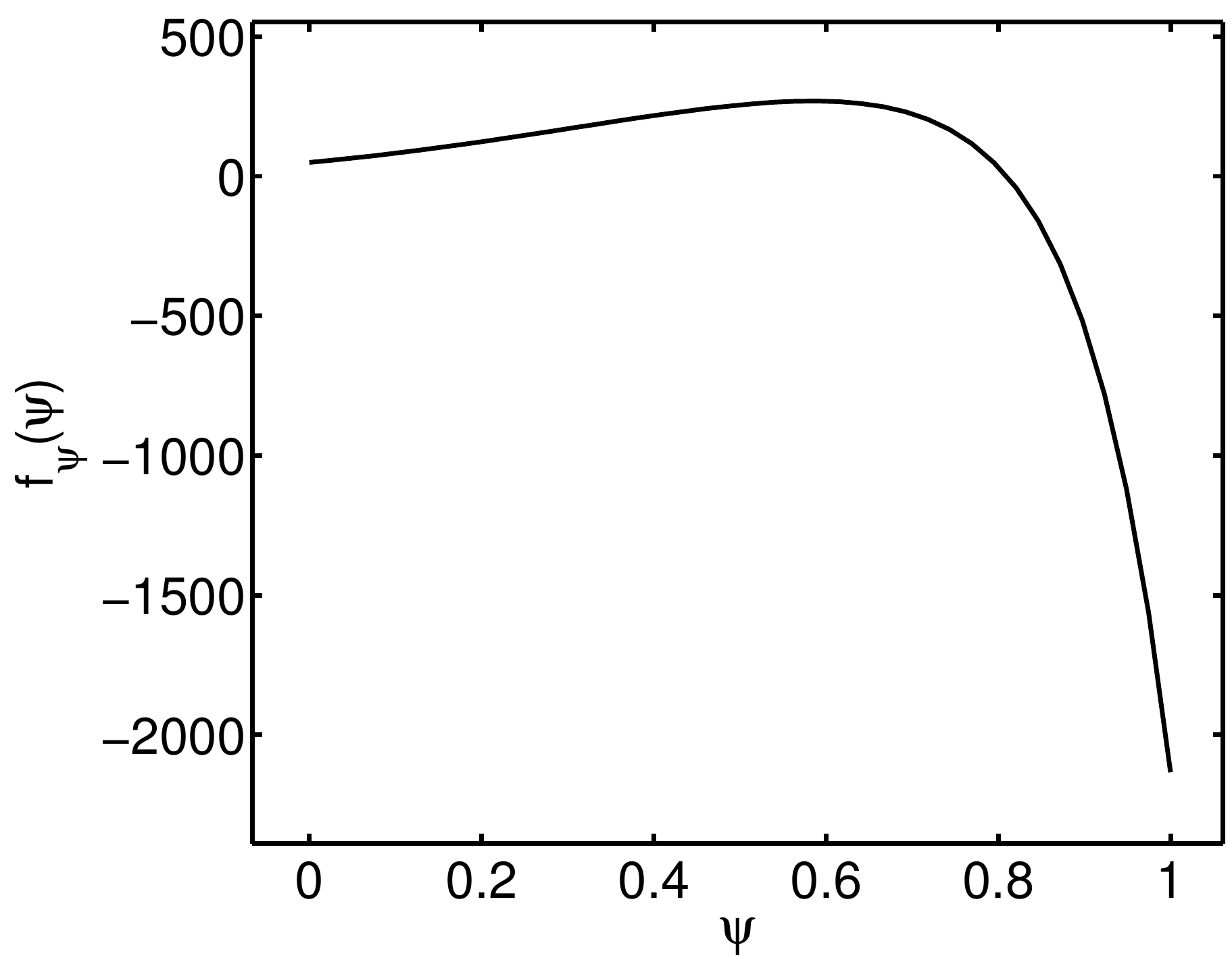}&
(c) \includegraphics[scale=0.2]{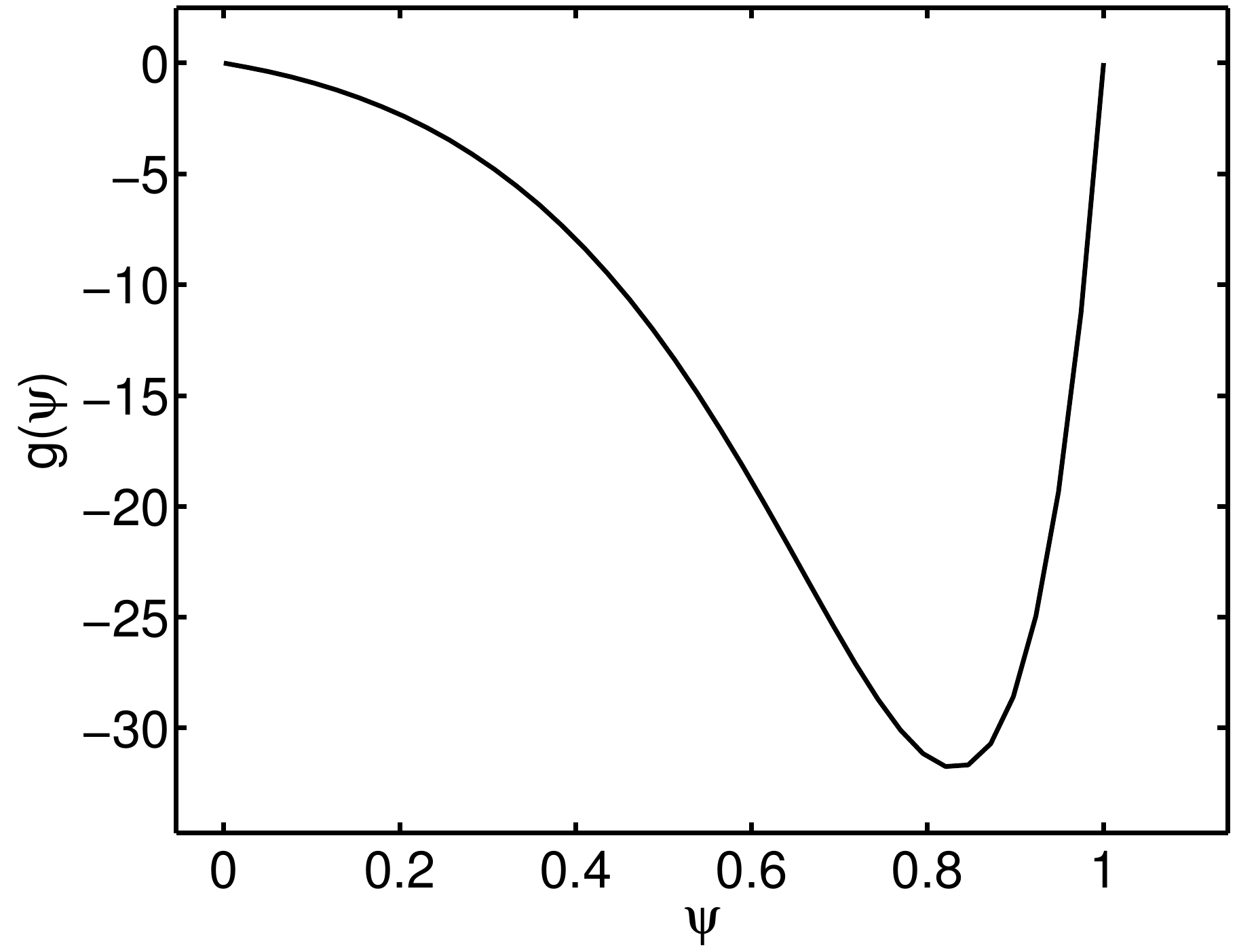} \\
(d) \includegraphics[scale=0.2]{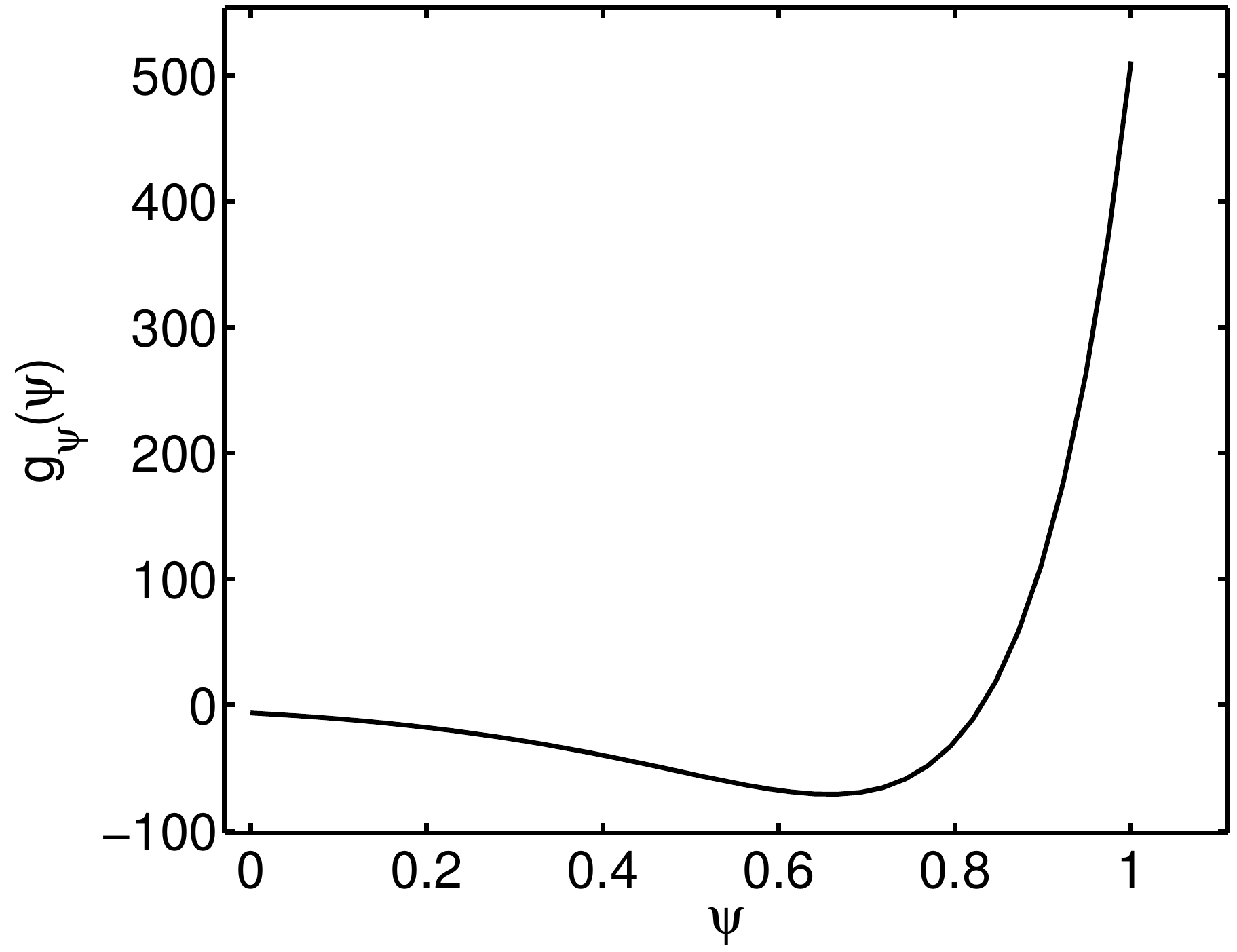}&
(e) \includegraphics[scale=0.2]{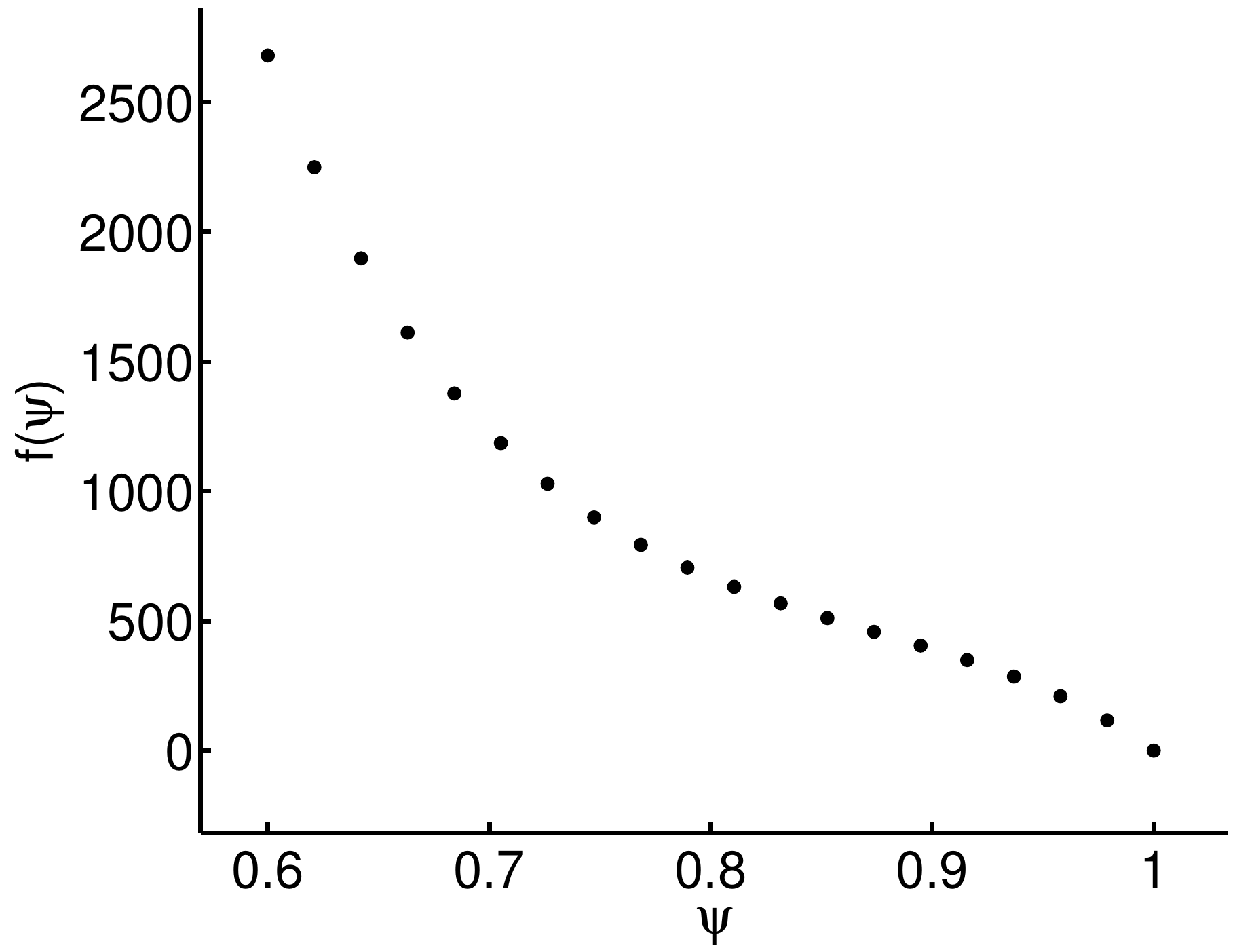} &
(f) \includegraphics[scale=0.2]{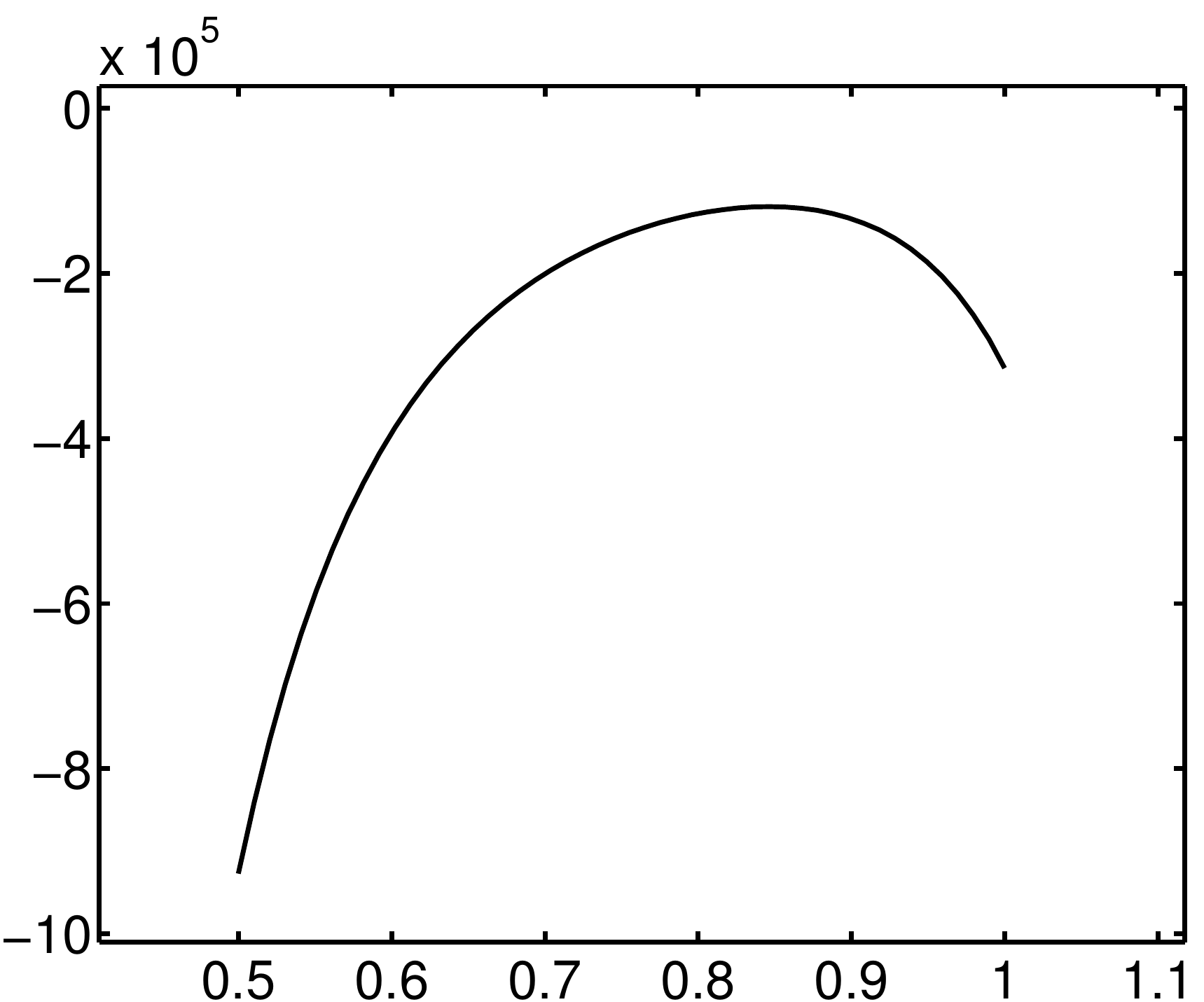}

\end{array}
$
\end{center}
\caption{(a) Plot of $\tilde f(\omega+ i\psi \omega')$ against $\psi$. (b) Plot of $\pd{}{\psi} \tilde f(\omega+ i\psi \omega')$ against $\psi$. (c) Plot of $\tilde g(\omega+i\psi \omega')$ against $\psi$, (d) Plot of $\pd{}{\psi} \tilde  g(\omega+ i\psi \omega')$ against $\psi$. (e) Plot of $\tilde f( i\psi \omega')$ against $\psi$, (f) Plot of $\pd{}{\psi} \tilde f(i\psi \omega')$ against $\psi$. As $\psi \to 0^+$, $\tilde f(i\psi \omega')\to \infty$.}
\label{fig96}
\end{figure}


\section{Stability for periods in the middle stability interval} \label{stabmid}

In this section we verify that conditions (A1), (A2), and (S1) hold for $k \in [0.9426,0.99999].$

\subsection{Simplicity of KdV eigenvalues}

In this section we show that for $k\in[0.942,0.9999984]$, corresponding to $X(k)\in[X_l,X_r]$ where $X_l \approx 8.43$ and $X_r \approx 26.07$, that the nonzero KdV eigenvalues of the linearized KdV operator are simple for all Floquet parameters $\xi \in [0,2\pi/X]$. By Lemma \ref{lemma:xi_zero}, $\lambda_{KdV} = 0$ only if $\xi = 0\mod \frac{2\pi }{X}$.
Recall that the KdV spectra are given by
$
\lambda_{KdV} = -4\wp'(\alpha),
$ where $\wp'$ is the derivative of the  Weierstrass elliptic function $\wp$, and $\alpha = n\omega + i\psi \omega'$.

\begin{lemma}\label{lemma:middle:simplicity}
The following hold for $k\in[0.942,0.9999984]$ and $\psi_0 = 0.95$: (1) $|\lambda_0(\omega+i\psi\omega')|<\lambda_0(i\psi_0\omega')$ for $\psi \in [-1,0]$, (2) $\xi(i\psi_0 \omega')<3\pi/X$, (3) the function $h(x,y)$ defined by $h(x,y):= (ci\lambda_0(\omega-ix\omega')-ci\lambda_0(iy\omega'))^2+(\omega\xi(\omega-ix\omega')-\omega\xi(iy\omega')-2\pi)^2$, where $c = -8\omega^3/(\pi\vartheta_1'(0))^3 $, is strictly convex for $(x,y)\in [0,1]\times [\psi_0,1]$, and (4) the nonzero KdV eigenvalues of the linearized KdV operator are simple for all Floquet parameters $\xi \in [0,2\pi/X]$.
\end{lemma}

\begin{proof}
(1)-(3) We use interval arithmetic to verify that $|\lambda_0(\omega+i\psi\omega')|< \lambda_0(i\psi_0\omega')$ for $\psi \in [-1,0]$, $\xi(i\psi_0 \omega')<3\pi/X$, and that $h_x(x,y) >0$ and $h_{xx}(x,y)h_{yy}(x,y)-h_{xy}^2(x,y) > 0$ for $(x,y)\in [0,1]\times [\psi_0,1]$. See section simplicity of \cite{Bdoc} for details. (4) By parts (1)-(3) and the Lemmata \ref{lemma:xi_zero}, \ref{lemma:properties}, and \ref{lemma:num_properties}, the nonzero KdV eigenvalues of the linearized KdV operator are simple for all Floquet parameters $\xi \in [0,2\pi/X]$. 
\end{proof}

\subsection{Distinctness of $\alpha_j$}

\begin{lemma}\label{lemma:middle:distinct}
For $k\in[0.9,0.9999995]$, condition (A2) holds. 
\end{lemma}

\begin{proof}
From the discussion in Section \ref{distinct}, it suffices to show that the $b_i(k)$ described in equation \ref{eq:distinct} are distinct in order to show that condition (A2) holds for a given $k$. We use interval arithmetic to verify the $b_i$ are distinct for $k\in[0.9,0.9999995]$. See Section distinct of \cite{Bdoc} for details.
\end{proof}

\subsection{Stability condition (S1)}\label{stabmid}

In this section we show that the stability condition (S1) holds for the middle stability region.

\begin{lemma}\label{lemma:middle:stability}
The stability condition (S1) holds for $k\in$[0.9426,0.9999983].
\end{lemma}

\begin{proof}
Recall the definition of $\Re(\lambda(\xi))$, $f(\alpha)$, and $g(\alpha)$ given in equation \eqref{innerprod}. In Section \ref{gspecial}, we show that $g(i\psi\omega')>0$ for $\psi\in(0,1)$, and in Section \ref{zeros} we show that $f(\alpha) = g(\alpha)=0$ for $\alpha \in \{\omega,\omega+i\omega',i\omega'\}$. For convenience, we reformulate the stability condition in \ref{formulation} and use analytic interpolation to evaluate $f(\alpha)$, $g(\alpha)$, $f_{\psi}(\alpha)$, and $g_{\psi}(\alpha)$ for other values of $\alpha = \tilde n \omega +i\psi \omega'$, $\tilde n \in \{0,1\}$, $\psi \in [0,1]$. We describe the method of interpolation in \ref{chebyshev}. We provide in Tables \ref{tableB}, \ref{tableC}, and \ref{tableD} details of the function bounds, stadium radii, and number of interpolation nodes used.

In the case $\alpha = \omega+i\psi \omega'$ we verify, using the interpolation polynomials, that $\Im{f(\alpha)}>0$ and $\Im{g(\alpha)}<0$ on the interval $\psi\in[(1+\cos(9\pi/10))/2,(1+\cos(\pi/10))/2]$, $f_{\psi}(\alpha)>0$ and $g_{\psi}(\alpha) < 0$ for $\psi \in [0,(1+\cos(9\pi/10))/2]$, and $f_{\psi}(\alpha)<0$ and $g_{\psi}(\alpha) > 0$ for $\psi \in [(1+\cos(\pi/10))/2,1]$. 

In the case that $\alpha = i\psi \omega'$, we factor (see \ref{factored}) $\tilde f(\alpha)$ in polynomial form $\tilde f(\alpha) = \sum_{k = 0}^5 \tilde f_k(\alpha) (i\xi)^k$ and verify that as a polynomial of $\xi$, $\tilde f(\alpha)$ has no roots in $\xi(\psi)$ for $\psi \in [0,10^{-3}]$. We then verify in factored form that $f(\alpha) > 0$ for $\psi \in [10^{-3},0.5]$. We use the factored form because $f(\alpha) \to \infty$ as $\psi \to 0$. For $\psi \in [0.5,1]$, we interpolate $f(\alpha)$ without factoring in $\xi$ and verify that $f(\alpha)>0$ for $\psi \in[0.5,0.9]$ and that $f_{\psi}(\alpha) <0$ for $\psi \in [0.9,1]$.

Together, these facts imply that $\Re(\lambda_1) < 0$ for $k\in [0.9426,0.99999]$. See Sections driver\_stability\_n0 and driver\_stability\_n1 of \cite{Bdoc} for details. 
\end{proof}

\begin{table}[!b]
\begin{tabular}{|c|c|c|c|c|c|c|c|c|c|c|}
\hline
$q_L$&$q_R$&$M_x$&$M_q$&$M_{\psi}$&$\rho_x$&$\rho_q$&$\rho_{\psi}$&$N_x$&$N_q$&$N_{\psi}$\\
\hline
0.1 & 0.4 & 2.78e+30 & 1.82e+28 & 1.29e+27 & 1.3 & 2.73 & 2.81 & 447 & 108 & 102\\
\hline
0.35 & 0.5 & 4.86e+29 & 9.14e+29 & 1.01e+28 & 1.22 & 10.1 & 5.57 & 581 &  48 &  72\\
\hline
0.49 & 0.538 & 6.26e+31 & 2.15e+35 & 2.06e+30 & 1.19 & 36.4 & 8.05 & 678 &  34 &  76\\
\hline
0.53 & 0.6 & 3.37e+34 & 4.76e+38 & 1.77e+33 & 1.16 & 22.3 & 9.02 & 867 &  42 &  80\\
\hline
0.59 & 0.66 & 3.46e+39 & 5.31e+44 & 2.1e+38 & 1.13 & 19.2 & 10.8 & 1.17e+03 &  49 &  88\\
\hline
0.65 & 0.71 & 3.32e+49 & 2.45e+55 & 2.25e+48 & 1.1 & 19.2 & 13.2 & 1.65e+03 &  57 & 104\\
\hline
\end{tabular}
\caption{Here we record the details of our analytic interpolation when $\alpha = i\psi\omega'$. Here $q_L$ and $q_R$ are respectively the left and right endpoints of the interval in $q$ on which we interpolate. The bounds on the functions for interpolation in the variables $x$, $q$, and $\psi$ are given respectively by $M_x$, $M_q$, and $M_{\psi}$. The number of interpolation nodes used are respectively $N_x$, $N_q$, and $N_{\psi}$. } 
\label{tableB}
\end{table}

\begin{table}[!b]
\begin{tabular}{|c|c|c|c|c|c|c|c|c|c|c|}
\hline
$q_L$&$q_R$&$M_x$&$M_q$&$M_{\psi}$&$\rho_x$&$\rho_q$&$\rho_{\psi}$&$N_x$&$N_q$&$N_{\psi}$\\
\hline
0.1 & 0.4 & 2.73e+30 & 1.8e+28 & 1.21e+27 & 1.3 & 2.73 & 2.81 & 446 & 108 & 102\\
\hline
0.35 & 0.5 & 3.04e+29 & 9.27e+29 & 4.61e+27 & 1.22 & 10.1 & 5.57 & 579 &  48 &  72\\
\hline
0.49 & 0.538 & 6.9e+30 & 2.09e+35 & 1.02e+29 & 1.19 & 36.4 & 8.05 & 666 &  34 &  74\\
\hline
0.53 & 0.6 & 1.02e+34 & 4.54e+38 & 2.66e+32 & 1.16 & 22.3 & 9.02 & 859 &  42 &  79\\
\hline
0.59 & 0.66 & 3.47e+39 & 5.15e+44 & 1.52e+38 & 1.13 & 19.2 & 10.8 & 1.17e+03 &  49 &  88\\
\hline
0.65 & 0.71 & 3.32e+49 & 2.44e+55 & 2.25e+48 & 1.1 & 19.2 & 13.2 & 1.65e+03 &  57 & 104\\
\hline
\end{tabular}
\caption{Here we record the details of our analytic interpolation when $\alpha =\omega+ i\psi\omega'$. Here $q_L$ and $q_R$ are respectively the left and right endpoints of the interval in $q$ on which we interpolate. The bounds on the functions for interpolation in the variables $x$, $q$, and $\psi$ are given respectively by $M_x$, $M_q$, and $M_{\psi}$. The number of interpolation nodes used are respectively $N_x$, $N_q$, and $N_{\psi}$. } 
\label{tableC}
\end{table}

\begin{table}[!b]
\begin{tabular}{|c|c|c|c|c|c|c|c|c|c|c|}
\hline
$q_L$&$q_R$&$M_x$&$M_q$&$M_{\psi}$&$\rho_x$&$\rho_q$&$\rho_{\psi}$&$N_x$&$N_q$&$N_{\psi}$\\
\hline
0.1 & 0.4 & 2.61e+29 & 4.43e+25 & 1.35e+26 & 1.3 & 2.73 & 2.81 & 437 & 102 & 100\\
\hline
0.35 & 0.5 & 2.87e+29 & 9.3e+27 & 2.74e+27 & 1.22 & 10.1 & 5.57 & 579 &  46 &  71\\
\hline
0.49 & 0.538 & 6.59e+30 & 3.14e+33 & 9.21e+28 & 1.19 & 36.4 & 8.05 & 665 &  33 &  74\\
\hline
0.53 & 0.6 & 9.98e+33 & 3.47e+37 & 2.52e+32 & 1.16 & 22.3 & 9.02 & 859 &  41 &  79\\
\hline
0.59 & 0.66 & 3.44e+39 & 1.46e+44 & 1.5e+38 & 1.13 & 19.2 & 10.8 & 1.17e+03 &  48 &  88\\
\hline
0.65 & 0.71 & 3.32e+49 & 2.04e+55 & 2.24e+48 & 1.1 & 19.2 & 13.2 & 1.65e+03 &  57 & 104\\
\hline
\end{tabular}
\caption{Here we record the details of our analytic interpolation of the factored form of the functions when $\alpha = i\psi\omega'$. Here $q_L$ and $q_R$ are respectively the left and right endpoints of the interval in $q$ on which we interpolate. The bounds on the functions for interpolation in the variables $x$, $q$, and $\psi$ are given respectively by $M_x$, $M_q$, and $M_{\psi}$. The number of interpolation nodes used are respectively $N_x$, $N_q$, and $N_{\psi}$.  } 
\label{tableD}
\end{table}


\section{Instability for periods in the lower instability region}\label{lowunstab}

In this section we describe our results showing that the periodic traveling-wave solutions of \eqref{gks},  corresponding to $k\in$[\kUnstableLowerLeft,\kUnstableLowerRight], described in Proposition \ref{p:kdvsolnexpand} are spectrally unstable.

\begin{lemma}\label{lemma:lower:instability}
For $k\in$[\kUnstableLowerLeft,\kUnstableLowerRight], corresponding to $X\in[X_l,X_r]$ where $X_l \approx$ \XUnstableLowerLeft and $X_r \approx $\XUnstableLowerRight, $\Re(\lambda_1(\omega+i\omega'))>0$; hence, by Proposition \ref{p:kdvstab} periodic traveling-wave solutions of \eqref{gks} described in Proposition \ref{p:kdvsolnexpand} are spectrally unstable for $\delta > 0$ sufficiently small. 
\end{lemma}

\begin{proof}
Recall from \eqref{innerprod2} that $\Re(\lambda_1(\alpha)) = \tilde f(\alpha)/\tilde g(\alpha)$ where $\alpha = \tilde n \omega + i\psi\omega'$. When $\psi = 1$,  $\tilde f(\alpha) = \tilde g(\alpha) = 0$ by Lemma \ref{lemma:zeros}; hence, $\Re(\lambda_1(\omega+i\omega')) = \tilde f_{\psi}(\omega+i\omega')/\tilde g_{\psi}(\omega+i\omega')$ so long as $\tilde g_{\psi}(\omega+i\omega')\neq 0$. We evaluate $\tilde f_{\psi}(\omega+i\omega')/\tilde g_{\psi}(\omega+i\omega')$ by interpolating $\tilde f_{\psi}$ and $\tilde g_{\psi}$ in the single variable $q$ with $\psi = \tilde n = 1$ fixed. See section \ref{chebyshev} for details of the method of interpolation. To keep the number of interpolation nodes needed small, we break up the interval in $q$ into subintervals. This increases the size of $\rho_q$ which significantly decreases the number of interpolation nodes needed. Using the interpolating polynomial with its error bounds, we use interval arithmetic to verify that $\Re(\lambda_1(\omega+i\omega')) > 0$ for $k \in[\kUnstableLowerLeft,\kUnstableLowerRight]$. The size of the subintervals in $k$ varies. See Table \ref{table:lower-instability} and  Section lower\_instability\_interpolation of \cite{Bdoc} for details. In Table \ref{table:lower-instability}, each row corresponds to a different interpolation polynomial. The first two columns indicate the range of $k$ values for which the interpolating polynomial was used and the third and fourth columns indicate the range of $q(k) = e^{-\pi K(\sqrt{1-k^2})/K(k)}$. The bounds on the modulus of the function being interpolated for each of the variables is given by $M_x$, $M_q$, and the radius of the stadium is given by $\rho_q$ and $\rho_{\psi}$. We used $N_x$ interpolation nodes in the variable $x$ and $N_{\psi}$ interpolation nodes in $q$. We give a lower bound on $\Re(\lambda_1)$ of $M_{\lambda}$. See Section \ref{chebyshev} and Tables \ref{tableB}, \ref{tableC}, and \ref{tableD} for interpolation details.
\end{proof}

\begin{table}[!b]
\begin{tabular}{|c|c|c|c|c|c|c|c|c|c|}
\hline
$k_L$&$k_R$&$q_L$&$q_R$&$M_x$&$M_q$&$\rho_q$&$N_x$&$N_q$&$M_{\lambda}$\\
\hline
0.9299 & 0.9422 & 0.122 & 0.139 & 1.53e+77 & 1.3e+18 & 27.6 & 130 & 25 & 4.41e-07 \\
\hline
0.89892 & 0.93008 &   0.1 &   0.2 & 1.24e+85 & 1e+18 & 5.26 & 140 & 51 & 0.342 \\
\hline
0.74993 & 0.90005 &  0.05 &  0.11 & 4.04e+74 & 1.11e+19 & 4.64 & 126 & 56 & 0.708 \\
\hline
0.39992 & 0.75108 &  0.01 &  0.06 & 8.14e+69 & 9.82e+20 & 2.18 & 120 & 118 & 0.399 \\
\hline
0.29993 & 0.40007 & 0.005 & 0.011 & 5.43e+63 & 1.03e+24 & 4.64 & 111 & 64 & 0.211 \\
\hline
0.23991 & 0.30009 & 0.003 & 0.006 & 2.58e+63 & 2.1e+25 & 5.26 & 111 & 61 & 0.14 \\
\hline
0.19991 & 0.24009 & 0.002 & 0.0038 & 2.52e+63 & 2.13e+26 & 5.67 & 111 & 59 & 0.0943 \\
\hline
\end{tabular}
\caption{In this table, entries in the first two columns are rounded to 5 significant digits and entries in the other columns are rounded to 3 significant digits. We verify instability for $k\in[k_L,k_R]$ using interpolating polynomials valid for $q\in[q_L,q_R]$. An interpolation bound in the variable $x$ with $\rho_x = 5.460277197252352$ for $\tilde f_{\psi}$ and $\tilde g_{\psi}$ is given by $M_x$ and in the variable $q$ by $M_q$. The number of interpolation nodes needed in $x$ and $q$ are respectively $N_x$ and $N_q$. The minimum of $\Re(\lambda_1)$ for $k\in[k_L,k_R]$ is bounded below by $M_{\lambda}$. The interpolation error is always less than 1e-16.}
\label{table:lower-instability}
\end{table}

\section{Instability for periods in the upper instability region}\label{upunstab}

In this section we describe our results showing that the periodic traveling-wave solutions of \eqref{gks},  corresponding to $k\in[0.99999839$,\kUnstableUpperRight], described in Proposition \ref{p:kdvsolnexpand} are spectrally unstable.

\begin{lemma}\label{lemma:upper:instability}
For $k\in[0.99999839,$\kUnstableUpperRight], corresponding to $X\in[X_l,X_r]$ where $X_l\approx 26.06$ and $X_r \approx$\XUnstableUpperRight, $\Re(\lambda_1(i\omega'))>0$ for some value of $\alpha$; hence, the periodic traveling-wave solutions of \eqref{gks} described in Proposition \ref{p:kdvsolnexpand} are spectrally unstable.
\end{lemma}

\begin{proof}
For $\alpha = i\psi \omega'$, we interpolate $\tilde f(\alpha)$ and $\tilde g(\alpha)$ given in equation \eqref{innerprod2} in the variables $q(k)$ and $\psi$. We then break the domain into $k$-intervals and use interval arithmetic to evaluate, via the interpolation polynomials, $\lambda_1(i\psi\omega')$ on each $k$ interval with 100 points in $\psi$ between 0.6 and 0.8. We verify that $\lambda_1(i\psi\omega') > 0$ for at least one of these $\psi$ values implying spectral instability. See Table \ref{tableB} and section driver\_instability\_upper of \cite{Bdoc} for details.
\end{proof}


\section{Determination of sharp stability transitions} \label{strict}

In this section we show that the stability transitions are sharp.

\begin{lemma}\label{transitionlemma}
Suppose that stability is determined by the sign of $h:[k_{\min},k_{\max}]\times [0,1]\to \R$, that is $h(k,\psi)<0$  for all $\psi\in[0,1]$ corresponds to stability at $k$ and $h(k,\psi)> 0$ for some $\psi$ corresponds to instability, where $h$ is analytic in both variables, and suppose that stability holds at one end of the interval $[k_{\max},k_{\min}]$ and not at the other. Further suppose that $h(k,\psi) < 0$ for $(k,\psi)\in [k_{\min},k_{\max}]\times ([0,1]/(\psi_1,\psi_2))$ where $\psi_1<\psi_2$. Then, if (i) $h_k \neq 0$ on $B:= [k_{\min},k_{\max}]\times [\psi_1,\psi_2]$, there is a sharp stability boundary which occurs for  $k_{\min} < k_*< k_{\max}$, (ii) If   $\mathrm{sign}(h_{\psi \psi})  =-\mathrm{sign}(h_k)$ (convexity), then there is a unique transition point $(k_*,\psi_*)$, with
$k_{\min} < k_*< k_{\max}$ and $\psi_1\leq \psi_*\leq \psi_2$, determined by the property $F:= (h,h_{\psi})=(0,0)$.
In this case, for any $(\psi,k)\in B$, we have
the a posteriori estimate
\be
\label{aposteriori}
|(k,\psi)-(k_*,\psi_*)|\leq   \sup_B  |dF^{-1}|  |F(k,\psi)|.
\ee
\end{lemma}

\begin{proof}
(i) Define the set $H:= \{(k,\psi) \in [k_{\min},k_{\max}]\times [0,1]| h(k,\psi) = 0\}$. Since $h$ is continuous and stability holds (without loss of generality) at $k_{\min}$ and does not hold at $k_{\max}$, $H$ is nonempty by the Intermediate Value Theorem. Define $k_* = \inf k:(k,\psi)\in H$. Since $h$ is continuous, $H$ is closed and so there exists $\psi_*$ such that $h(k_*,\psi_*) = 0$. By the Mean Value Theorem, $h(k,\psi_*) > 0$ for $k>k_*$ implying instability, and by definition of $k_*$, stability holds for $k<k_*$. By continuity of $h$, $k_{\min}<k_*<k_{\max}$. (ii) Suppose without loss of generality that stability holds at $k_{\min}$ so that $h_k > 0$ in $B$ and let $k_*$ be as in (i) . Suppose that $h(k_*,\psi_1) = h(k_*,\psi_2) = 0$. By convexity in $\psi$, $h(k_*,(\psi_1+\psi_2)/2) > 0$. But $h(k_{\min},(\psi_1+\psi_2)/2) < 0$ so by the Intermediate Value Theorem, there is a $k_0< k_*$ such that $h(k_0,(\psi_1+\psi_2)/2) = 0$. This contradicts the definition of $k_*$. Taylor expanding about $(k_*,\psi_*)$ we have $F(k,\psi) = F_k(k_0,\psi_0)(k-k_0)+F_{\psi}(k_0,\psi_0)(\psi-\psi_0)$ for some $(k_0,\psi_0)$. Then $|(k,\psi)-(k_*,\psi_*)|\leq   \sup_B  |dF^{-1}|  |F(k,\psi)|$.
\end{proof}

\begin{remark}
In principle, one may approximate $(k_*,\psi_*)$ to any given precision via Newton's method with multiple precision arithmetic and the evaluation of $F$ in the error bound with multiple precision interval arithmetic, but we do not do so here.  
\end{remark}

\subsection{Sharp transition at the lower stability boundary}

In this section we verify that the transition from instability to stability is strict at the lower stability boundary. Define $J:= [0.942197747747748,0.9426]$ and the regions $R_1 := J\times[0,(1+\cos(9\pi/10))/2]$, $R_2 := J\times [(1+\cos(9\pi/10))/2,(1+\cos(\pi/5))/2]$, $R_3 := J\times[(1+\cos(\pi/5))/2,0.99]$, and $R_4:= J\times[0.99,1]$. 

\begin{lemma}\label{lemma:lower:strict1}
For $(k,\psi) \in R_4$ and $\alpha = \omega+i\psi\omega'$, $\pd{}{k}\lambda_1(k,\psi) < 0$.
\end{lemma}

\begin{proof}
Recall that $\lambda_1 =\tilde f/\tilde g$ where $\tilde f$ and $\tilde g$ are defined in equation \eqref{lemma:zeros}. Dropping the tilde notation, $\pd{}{k}\lambda_1 = f_k/g - (f/g)(g_k/g)$. When $\psi = 1$, $f(k,\psi) = g(k,\psi) = 0$ for all $k\in(0,1)$ by \ref{lemma:zeros}, so that Taylor expanding $f$ about $\psi_0 = 1$ and simplifying yields $f(k_0,\psi) = f_{\psi}(k_0,1)(\psi-1) -\int_{\psi}^1f_{\psi\psi}(k_0,t)(\psi-1)dt$. We similarly Taylor expand $g$, $\pd{}{k}f$, and $\pd{}{k}g$. Let $M_1$ such that $|f_{\psi\psi}(k_0,\psi)|\leq M_1$ for $(k,\psi)\in R_{4}$ and similarly let $M_2$, $M_3$, and $M_4$ be bounds on $g_{\psi\psi}$, $f_{k\psi\psi}$, and $g_{k\psi\psi}$ respectively. Note that $|\int_{\psi}^1f_{\psi\psi}(k_0,t)(\psi-t)dt|\leq M_1(\psi-1)^2/2$. Defining the interval $z = (1/2)[-0.01,0.01]$, we have upon substitution and simplification that
\begin{equation}
\pd{}{k}\lambda_1(k_0,\psi) \in \frac{f_{\psi k}(k_0,1)+M_3z}{g_{\psi}(k_0,1)+M_2z} -\left(\frac{f_{\psi}(k_0,1)+M_1z}{g_{\psi}(k_0,1)+M_2z}\right)\left(\frac{g_{k\psi}(k_0,1)+M_4z}{g_{\psi}(k_0,1)+M_2z} \right).
\label{eq:lower:strict}
\end{equation}
 Recall that $ f = f_1+f_2/\omega^2$. We approximate $f_1$, $f_2$, $g$ and their first two derivatives in $\psi$ using Chebyshev interpolation in the variables $\tilde q\in[-1,1]$ and $\tilde \psi\in[-1,1]$ where $q = 5/2+3\tilde q/2$, $q = e^{-\pi K(\sqrt{1-k^2})/K(k)}$, and $\psi= (1+\tilde \psi)/2$; see Section \ref{chebyshev} for interpolation details. We take $\psi \in [0,1]$, $q\in[0.1,0.4]$, $\rho_{\psi} \approx 2.81$ and $\rho_q \approx 2.733$. We found function bounds of $M_q \approx 2.99\times 10^{37}$ for $q\in E_{q}$, $M_{\psi} \approx 1.08\times 10^{35}$ for $\psi \in E_{\psi}$, and $M_x \approx 3.12\times 10^{38}$ for $x\in E_{x}$. We use $N_q = 205$ points to interpolate in the variable $q$, $N_{\psi} = 193$ nodes in $\psi$, and $N_x = 521$ nodes in $x$. The associated one dimensional approximation error for each of the variables is respectively $err_q \approx 3.81 \times 10^{-50}$, $err_{\psi} \approx 9.62 \times 10^{-50}$, and $err_x \approx 9.78\times 10^{-18}$. We intentionally make the interpolation error in the variables $\tilde q$ and $\tilde \psi$ excessively small to provide good error bounds when approximating the derivative of the function with the derivative of the interpolant. Now $\pd{}{k}g =( \pd{}{k} q )(\pd{}{q}\tilde q)( \pd{}{\tilde q}g)$. To approximate $\pd{}{\tilde q}g(\tilde q, \tilde \psi)$, we use $\pd{}{\tilde q} p(\tilde q,\tilde \psi)$ where $p$ is the Chebyshev interpolant. To bound the approximation error, we need a bound on $\pd{}{\tilde q}g$ on a stadium, see \ref{dererror}. We take $\rho_{\check q} \approx 1.75$ and use Cauchy's integral formula to determine that $|\pd{}{\tilde q}g| \leq M \approx 10^{39}$ for $q\in E_{\check q}$. Putting these together, the error of approximating $\pd{}{\tilde q}g$ with $N_q = 205$ interpolation nodes is bounded by $err_{\check q} \approx 10^{-9}$. With these error estimates in place, we are able to evaluate \eqref{eq:lower:strict}. We break $J$ up into 5 convenient sub-intervals to verify $\pd{}{k}\lambda_1(k,\psi) < 0$. For example,   $\pd{}{k}\lambda_1(k,\psi)\in[r_-,r_+]$ where $r_{\pm} = -34.27388757999935\pm 13.49688667120629$ when we use the sub-interval $[0.9425,0.9426]$. 
\end{proof}

The computations in \eqref{eq:lower:strict} are of size: $f_{\psi,k}=O(10^4)$, $f_{\psi}=O(10^3)$, $g_{\psi}=O(10^2)$, $g_{\psi k} = O(10^3)$, $f_{\psi\psi k} = 0(10^6)$, $g_{\psi\psi k} = O(10^5)$, $f_{\psi\psi} = O(10^3)$, and $g_{\psi\psi} = O(10^4)$. This indicates the difficulty of the computation since the interval in $\psi$ must be taken sufficiently small, about $10^{-2}$ in width, to avoid loss of information.

\begin{lemma} \label{lemma:lower:strict2}
When $\alpha = \omega + i\psi \omega'$, $\tilde f > 0$ for $(k,\psi)\in R_{1}\cup R_{2}$ and $\pd{}{\psi} \tilde f > 0$ for $(k,\psi)\in R_{3}$. In addition, $\tilde g< 0$ for $(k,\psi)\in J\times [0,1]$.
\end{lemma}
 
\begin{proof}
Recall that $\lambda_1 =\tilde f/\tilde g$ where $\tilde f$ and $\tilde g$ are defined in equation \eqref{innerprod2}. We use analytic interpolation with error bounds on 5 sub-intervals of $J$ to approximate $\tilde f$, $\tilde g$, and their derivatives with respect to $\psi$. The interpolation details are the same as those described in Section \ref{stabmid}. We verify with interval arithmetic via evaluation of the Chebyshev interpolants that $f_{\psi} > 0$ for $(k,\psi)\in R_{1}$ and that $\tilde f > 0$ for $(k,\psi)\in R_{2}$. By Lemma \ref{lemma:zeros}, $\tilde f(k,0) = 0$; hence  $f(k,\psi) > 0$ for $(k,\psi)\in R_{1}$. We verify that $\pd{}{\psi} \tilde f > 0$ for $(k,\psi)\in R_{3}$. For $(k,\psi)\in R_{1}$, we verify that $g_{\psi} < 0$, for $(k,\psi)\in R_{2}$ we verify $\tilde g < 0$, and for $(k,\psi)\in R_{3}\cup R_{4}$ we verify that $\tilde g_{\psi} > 0$. This together with Lemma \ref{lemma:zeros} verifies that $\tilde g < 0$ on the indicated region. 
\end{proof}

We are now ready to show strict transition of stability at the lower stability boundary.

\begin{lemma} \label{lemma:lower:strict3}
There exists a unique $k_*\in J$ where $J:= [0.942197747747748,0.9426]$ such that the stability condition (S1) holds for $\{k>k_*\}\cap J$ and (S1) does not hold for $\{k<k_*\}\cap J$. \end{lemma}

\begin{proof}
In Section \ref{stabmid}, we showed that $\lambda_1 < 0$ for $k\in J$ and $\psi \in [0,1]$ in the case $\alpha = i\psi \omega'$. Consider the case $\alpha = \omega+i\psi \omega'$. By Lemma \ref{lemma:lower:strict2}, $\lambda_{1} < 0$ for $(k,\psi)\in R_{1}\cup R_{2}$. From Lemma \ref{lemma:lower:strict2}, we also have that $\lambda_{1} >0$ on $R_{3}\cup R_{4}$ if and only if $\tilde f <0$. Suppose $\tilde f < 0$ for some $(k_{0},\psi_{0}) \in R_{3}$. By Lemma \ref{lemma:lower:strict2}, $\pd{}{\psi} \tilde f<0$ in $R_{3}$ which implies $\tilde f(k_{0},0.99)\leq 0$. We verified (S1), which depends on the sign of $\tilde f$, holds for $k =0.9426$ in Section \ref{stabmid}, so by Lemma \ref{lemma:lower:strict1}, $\tilde f(k_{0},0.99)\leq 0$ implies there is a $k_{1}\in J$ with $k_{1}\geq k_{0}$ such that $\tilde f(k_{1},0.99)= 0$. Thus, condition (S1) holds for $k\in J$ if and only if $\lambda_{1}<0$ for $(k,\psi)\in R_{4}$. Since condition (S1) does not hold for $k =0.942197747747748$ and does hold for $k =0.9426$ as verified in Sections \ref{lowunstab} and \ref{stabmid} respectively, we have by lemma \ref{lemma:lower:strict1} and the application of \ref{transitionlemma} to $\lambda_1$ that the stability transition is strict. 
\end{proof}

%
%

\subsection{Sharp transition at the upper stability boundary}

In this section we show that the transition from stability to instability is sharp at the upper stability boundary. Define $J:= [0.9999983,0.99999839]$, $R_1:= J\times [0,0.001]$, $R_2:= J\times [0.001,0.5]$, $R_3:= J\times [0.5,0.7]$, $R_4:= J\times [0.7,0.8]$, $R_5:= J\times [0.8,0.9]$, and $R_6:= J\times [0.9,1]$.  

\begin{lemma}\label{lemma:upper:strict1}
For $\alpha =  i\psi \omega'$ and $(k,\psi)\in R_1\cup R_2\cup R_3 \cup R_5$, $\tilde f > 0$. For $(k,\psi)\in R_6$, $\tilde f_{\psi} < 0$ and $\tilde f > 0$ for $\psi < 1$, and for $(k,\psi)\in R_4$, $\tilde f_k < 0$ and $\tilde f_{\psi\psi} > 0$. 
\end{lemma}

\begin{proof}
We verify that $\tilde f > 0$ for $(k,\psi)\in R_1$, $R_2$, $R_3$, and $R_5$  and that $f_{\psi} <0$ for $(k,\psi)\in R_6$ in the same way we do in Section \ref{stabmid}, except that we have removed the interval $R_4$ from the computation. Since $\tilde f(k,1) = 0$ and $f_{\psi} <0$ in $R_6$, $\tilde f > 0$ for $(k,\psi)\in R_6$, $\psi < 1$.  We break $J$ up into 1500 sub-intervals of two different lengths and we break up the intervals in $\psi$, corresponding to the different regions $R_j$, into 1200-5000 sub-intervals. To verify $\tilde f_k < 0$ and $\tilde f_{\psi\psi} > 0$ for $(k,\psi)\in R_4$, we interpolate $f$ and $f_{\psi}$ and then approximate their derivatives with respect to $\psi$ using the derivative of the interpolating polynomial. We break $J$ into 1001 subintervals for this computation and take $100$ subintervals in $\psi$. The interpolation error bounds for the derivatives are described in Section \ref{chebyshev}. To bound the error, we need an upper bound on $|f_k|$ and $|f_{\psi}|$ on a stadium. We use Cauchy's integral formula on a stadium $E_{\rho_1}$ to obtain the desired bound on the derivatives on a stadium $E_{\rho_2}$ with $1< \rho_2<\rho_1$. 
\end{proof}

We are now ready to prove the following lemma.
\begin{lemma}\label{lemma:upper:strict2}
There exists a unique $k_*\in J$ where $J:= [0.9999983,0.99999839]$ such that the stability condition (S1) holds for $\{k<k_*\}\cap J$ and (S1) does not hold for $\{k>k_*\}\cap J$. Further, there is a unique $\psi_*\in[0,1]$ such that $\lambda_1(k_*,\psi_*) = 0$. 
\end{lemma}

\begin{proof} We verify as in Section \ref{stabmid}, that $\lambda_1 < 0$ for $\alpha = \omega + i\psi \omega'$. Consider the case $\alpha = i\psi \omega'$. From Lemma \ref{gspecial}, we have that $\tilde g < 0$ for $\psi\in (0,1)$. Thus, the stability condition (S1) depends on the sign of $\tilde f$ with $-\tilde f < 0$ corresponding to stability. We verify in Section \ref{stabmid} that (S1) holds at $k = 0.9999983$, and in Section \ref{upunstab} that (S1) does not hold at $k = 0.99999839$. The Lemma then follows from \ref{lemma:upper:strict1} and the the application of Lemma \ref{transitionlemma} to $-\tilde f$.
\end{proof}

Now that we know the transition from stability to instability is strict, we approximate the transition to higher accuracy.

\begin{lemma}
For $k =0.999998385205026$ and $\psi\in[0.7,0.8]$ where $\alpha = i\psi\omega'$, $\tilde f > 0$. When $k = 0.999998385263233$, $\tilde f<0$ for some value of $\psi\in[0.7,0.8]$.
\label{lemma:strict:upper:pinpoint}
\end{lemma} 

\begin{proof}
When $k =0.999998385205026$ and $\alpha = i\psi\omega'$, we verify directly (without analytic interpolation) with interval arithmetic that $\tilde f(i\psi\omega') > 0$ for $\psi \in [0.7,0.8]$. When $k = 0.999998385263233$, we verify directly that $\tilde f < 0$ for some value of $\psi\in[0.7,0.8]$. See Section upper\_pinpoint of \cite{Bdoc} for details.
\end{proof}

{\textbf Remark.} We used a bisection scheme to chose the values of $k$ to verify in Lemma\ref{lemma:strict:upper:pinpoint}; see Section upper\_pinpoint of \cite{Bdoc} for details.


\section{Proof of the Main Theorem} \label{mainproof}

Collecting the above results, we may now give a proof of the main theorem stated in Section \ref{mainresult}.

\begin{proof}[Proof of Theorem \ref{thm:main}]
The stability condition (S1) does not hold for 
\begin{center}
$k\in H_l:=$ [\kUnstableLowerLeft,\ \kUnstableLowerRight] 
\end{center}
or
\begin{center}
 $k\in H_r:= [$0.99999839,\ \kUnstableUpperRight] 
 \end{center}
by lemma \ref{lemma:lower:instability} and \ref{lemma:upper:instability} respectively. By Lemma \ref{lemma:middle:stability}, the stability condition (S1) holds for $k\in H_m:=$ [\kStableLeft,0.9999983]. The transition of the stability condition (S1) is strict between $H_l$ and $H_m$ and between $H_m$ and $H_r$ by lemma \ref{lemma:lower:strict3} and \ref{lemma:upper:strict2} respectively. By the Lemmata \ref{lemma:upper:strict1} and \ref{lemma:upper:strict2} and their proofs, Lemma \ref{lemma:strict:upper:pinpoint} implies that the stability condition (S1) holds for $k \in [\sup H_m,0.999998385205026]$ and does not hold for $k\in [ 0.999998385263233,\inf H_u]$.

In Lemma \ref{lemma:middle:simplicity} we verify that assumption (A1) holds for $k\in[0.942,0.9999984]$ and in Lemma \ref{lemma:middle:distinct} we verify that assumption (A2) holds for $k\in[0.9,0.9999995]$. Hence, by Proposition \ref{p:kdvstab}, the periodic traveling-wave solutions of \eqref{gks} described in Proposition \ref{p:kdvsolnexpand} are spectrally, thus nonlinearly, stable for $k\in[k_l,k_r]$ where $\sup H_l < k_l< \inf H_m$ and $0.99999838520< k_r < 0.99979838526$.
\end{proof}

This completes our numerical verification of stability of periodic traveling-wave solutions of the Kuramoto-Sivashinsky equation in the Korteweg-de Vries limit.

\medskip
{\bf  Acknowledgment:} We thank Mathew Johnson, Norm Levenberg, and especially Kevin Zumbrun for helpful conversations. All computations were carried out using STABLAB developed by Jeffrey Humpherys, Kevin Zumbrun, and the author. Interval arithmetic computations were carried out using version 6 of INTLAB developed by Siegfried M. Rump, head of the Institute for Scientific Computing
at the Hamburg University of Technology, Germany.  
%
%
\appendix

\section{Computing environment}

All computations were carried out in at least one of the two computational environments described below.

Some computations were carried out on a MacBook laptop with 2GB memory and a duo core Intel processor with 2GHz processing speed. The software environment consists of MAC OS X 10.5.8, MATLAB R2008b, STABLAB 2013, and INTLAB 6.

Some computations were carried out on a 2009 Mac Pro with 16GB memory and two quad-core intel processors with 2.26 GHz processing speed. The software environment consists of OS X 10.8.4, MATLAB R2013b, STABLAB 2013, and INTLAB 6.

\section{Weierstrass functions}

Here we gather some useful properties of the Weierstrass functions and the theta function. The Weierstrass functions include the Weierstrass elliptic function $\wp(z)$ and it's derivative $\wp'(z)$, the Weierstrass zeta function $\zeta(z)$, and the Weierstrass sigma function $\sigma(z)$. See \cite{WR} and \cite{GR} for more properties.

%
%


%

The Weierstrass sigma functions have the following representations in terms of the theta functions:
\begin{equation}
\begin{split}
\sigma(z)&= \frac{2\omega \vartheta_1(\pi z/2\omega)}{\pi \vartheta_1'(0)}e^{\eta_1z^2/2\omega},\quad
\zeta(z)= \frac{\eta_1 z}{\omega} + \frac{\pi\vartheta_1'(\pi z/2\omega)}{2\omega\vartheta_1(\pi z/2\omega)}\\
\wp(z)&= \wp(\omega) + \left( \frac{\pi \vartheta_1'(0)\vartheta_2(\pi z/2\omega)}{2\omega \vartheta_2(0)\vartheta_1(\pi z/2\omega) } \right)^2 ,\quad
\wp'(z)= \frac{-\pi^3\vartheta_1'(0)^2}{4\omega^3}\frac{\vartheta_2(\pi z/2\omega)\vartheta_3(\pi z/2\omega) \vartheta_4(\pi z/2\omega)}{\vartheta_1^3(\pi z/2\omega)}.
\end{split}
\label{w_of_theta}
\end{equation}

%
%

\subsection{The q-series representation} The Weierstrass elliptic functions have the following q-series representations:

\begin{equation}
\begin{split}
q&= e^{-\pi  \omega'/\omega},\quad
\eta_1  = \zeta(\omega),\quad
\eta_1= \frac{\pi^2}{12 \omega}-\frac{2\pi^2}{\omega}\sum_{k=1}^{\infty} \frac{kq^{2k}}{1-q^{2k}}\\
p(z) &= \frac{-\eta_1}{\omega} + \left(\frac{\pi}{2\omega}\right)^2\csc^2\left(\frac{\pi z}{2\omega}\right) - \frac{2\pi^2}{\omega^2}\sum_{k=1}^{\infty} \frac{kq^{2k}}{1-q^{2k}}\cos\left(\frac{k\pi z}{\omega}\right)\\
p(z+i\omega')&= -\frac{\eta_1}{\omega}-\frac{2\pi^2}{\omega^2}\sum_{k=1}^{\infty}\frac{kq^k}{1-q^{2k}}\cos\left(\frac{k\pi z}{\omega}\right)\\
\vartheta_1(z) &= 2\sum_{n=1}^{\infty} (-1)^{n+1}q^{(n-1/2)^2}\sin((2n-1)z)\\
p(z+\omega)&= -\frac{\eta_1}{\omega}+\left(\frac{\pi}{2\omega}\right)^2\sec^2\left(\frac{\pi z}{2\omega}\right) - \frac{2\pi^2}{\omega^2}\sum_{k=1}^{\infty} (-1)^k
\frac{kq^{2k}}{1-q^{2k}} \cos\left(\frac{k\pi z}{\omega}\right)\\
p'(z) &= \frac{-\pi^3}{4\omega^3}\cot\left(\frac{\pi z}{2\omega}\right) \csc^2\left(\frac{\pi z}{2\omega}\right)+ \frac{2\pi^3}{\omega^3}\sum_{k=1}^{\infty} \frac{k^2q^{2k}}{1-q^{2k}}\sin\left(\frac{k\pi z}{\omega}\right)\\
p'(z+i\omega')&= \frac{2\pi^3}{\omega^3}\sum_{k=1}^{\infty} \frac{k^2q^k}{1-q^{2k}}\sin\left(\frac{k\pi z}{\omega}\right)\\
\zeta(z)&= \frac{\eta_1 z}{\omega} + \frac{\pi }{2\omega} \cot\left(\frac{\pi z}{2\omega}\right) + \frac{2\pi }{\omega}\sum_{k=1}^{\infty} \frac{q^{2k}}{1-q^{2k}}\sin\left(\frac{k\pi z}{\omega}\right)\\
\sigma(z) &= \frac{2\omega}{\pi} + \sin\left(\frac{\pi z}{2\omega}\right) + \exp\left( \frac{\eta_1 z^2}{2\omega} + 4\sum_{k=1}^{\infty} \frac{q^{2k}}{k(1-q^{2k})}\sin^2\left(\frac{k\pi z}{2\omega}\right)\right).
\end{split}
\label{qseries}
\end{equation}

%
%

\subsection{Properties}\label{properties}
For $m,n \in \Z$, the following hold:
\begin{equation}
\begin{split}
p(z+2m\omega+2ni\omega') &= p(z),\quad
p'(z+2m\omega+2ni\omega') = p'(z),\\
\sigma(z+2m\omega + 2ni\omega')&= (-1)^{mn+m+n}e^{2(m\zeta(\omega)+n\zeta(i\omega'))(z+m\omega+ni\omega')}\sigma(z),\\
\zeta(z+2m\omega+2ni\omega')&= \zeta(z)+2m\zeta(\omega)+2n\zeta(i\omega').
\end{split}
\label{prop:periodicity}
\end{equation}

%
%
The following oddness properties hold: 
$
p'(-z) = -p'(z),\
\sigma(-z) = -\sigma(z),\
\zeta(-z) = -\zeta(z).
$

%
%

The following addition formulae hold:
\begin{equation}
\begin{split}
\sigma(z_1+z_2)\sigma(z_1-z_2)&= -\sigma^2(z_1)\sigma^2(z_2)(\wp(z_1)-\wp(z_2)),\\
\zeta(z_1+z_2)&= \zeta(z_1)+\zeta(z_2)+\frac{1}{2}\frac{p'(z_1)-p'(z_2)}{p(z_1)-p(z_2)}.
\end{split}
\label{prop:addition}
\end{equation}

The following identities hold:
$
\zeta(i\omega')\omega - i\omega'\zeta(\omega)= -\frac{\pi i}{2},
$
$
\overline{p(z)}= p(\bar z)$,
$\overline{p'(z)}= p'(\bar z)$,
$
p'(\omega)=0$,
$p'(i\omega')=0$,
$p(\omega)= \frac{3g_3}{g_2}$,
$p(i\omega')=-\frac{3g_3}{g_2}$.

%
%

The Weierstrass elliptic functions have the following relationship with respect to their derivatives:
$
\frac{\partial}{\partial z} \wp(z)= \wp'(z)$,
$\frac{\partial}{\partial z} \wp'(z)= 6\wp^2(z)-\frac{g_2}{2}$,
$\frac{\partial}{\partial z} \sigma(z)= \sigma(z)\zeta(z)$,
$\frac{\partial}{\partial z} \zeta(z)= -\wp(z)$,
where $g_2$ and $g_3$ are the Weierstrass elliptic function invariants.

\section{Stability for large $\xi$}

Note that $\zeta(z)$ is analytic in a neighborhood containing the vertical line $l_{k} := \{  z = \omega + i\beta\ |\  \beta \in [0, \omega']\}$. Hence, $\xi(\alpha)$ is bounded for $\alpha \in l_k$.

 We now show that in the case $\alpha = i\beta$, $\beta \in (0,\omega']$, that $\xi(\alpha)$ can not correspond to $\lambda_1 > 0$ for $\beta \in (0,\rho)$ for some $\rho > 0$. Recall that $\xi(i\beta)\to \infty$ as $\beta \to 0+$ and $\xi(i\beta)$ is decreasing as a function of $\beta$.
Remembering the definition of $w(x,\alpha)$ from \eqref{vandder}, we note that $w(x,0)=1$ and $w(x,\alpha)$ is analytic in a region of $\alpha = 0$. Hence, there exists $\rho >0$ and $M_L>0$, $M_k>0$, $k = 0,1,...,4$ such that whenever $0<\beta \leq \rho$, $|w(x,i\beta)| \geq M_L$ and $|w^{(k)}(x,i\beta)| \leq M_k$ for all $x\in [0,X]$. Letting $y(\tilde x) := X/2 + X\tilde x/2$, we see

\eq{
\int_{-1}^1 v(y)\bar v'(y) d\tilde x &= \int_{-1}^1 w(y)(\bar w'(y)-i\xi \bar w(y))dx,\\
&= \int_{-1}^1 w(y)\bar w'(y) - (i\xi)|w(y)|^2d\tilde x,
}{}
and so 

\eq{
\Im\left(\int_0^X v(x)\bar v'(x) dx \right)& \leq 2( M_0M_1 - \xi M_L^2).
}{\label{numer_bound}}

Similarly,

\eq{
\Im\left(\int_{-1}^1 (v''(y)+v''''(y))\bar v'(y) d\tilde x\right)&\leq -M_L^2\xi^5 + \xi^4M_0M_1 +\\
& (M_1+M_0\xi)(4M_1\xi^3+6M_2\xi^2+4M_3 \xi + M_4 + M_0 \xi^2 + 2M_1 \xi + M_2).
}{\label{denom_bound}}

Let $\xi_0> 0$ such that the right hand side of \eqref{numer_bound} and \eqref{denom_bound} are negative for $\xi \geq \xi_0$. Let $\rho :=-i \xi^{-1}(\xi_0)$. Note that $\lambda_1(i\rho) < 0$ Then for $0<\beta\leq \rho$, $\xi(i\beta)$ does not correspond to instability.




\begin{thebibliography}{BHLyZ1}

\bibitem{Ba} N. K. Basu, 
{\it On Double Chebyshev Series Approximation},
SIAM Journal on Numerical Analysis, V 10 No. 3, pg 496-505, 1973.

\bibitem{BN} Doron E. Bar and Alexander A. Nepomnyashchy,
{\it Stability of periodic waves governed by the modified Kawahara equation},
Physica D, V 86, pg 586-602, 1995.

\bibitem{Bdoc} B. Barker,
{\it STABLAB Documentation for KdV : Numerical proof of stability of roll waves in the small-amplitude limit for inclined thin film flow}, Available at \url{http://arxiv.org/archive/math}. See also \url{http://pages.iu.edu/~kzumbrun/}.

\bibitem{BJNRZ0}
B. Barker, M.A. Johnson, P. Noble, L.M. Rodrigues, K. Zumbrun.
{\it Stability of Periodic Kuramoto-Sivashinsky Waves }, Applied Math. Letters 25 (2012) no 5, 824-829.

\bibitem{BJNRZ1}
B. {Barker}, M. A. {Johnson}, P. {Noble}, L. M. {Rodrigues}, and
  K. {Zumbrun}.
\newblock Nonlinear modulational stability of periodic traveling-wave solutions of the generalized Kuramoto-Sivashinsky equation.
\newblock \emph{Physica D}, 258 no. 1: 11-46, 2013.

\bibitem{BJNRZ2} B.~Barker, M. Johnson,
P. Noble, M. Rodrigues, and K.~Zumbrun,
{\it Whitham averaged equations and modulational stability
of periodic solutions of hyperbolic-parabolic balance laws,}
Journ\'ees \'equations aux d\'eriv\'ees partielles (2011), Exp. No. 3.

\bibitem{BJNRZ4} B.~Barker, M. Johnson,
P. Noble, M. Rodrigues, and K.~Zumbrun,
{\it Efficient numerical evaluation of the periodic Evans function of Gardner 
and spectral stability of periodic viscous roll waves}, in preparation.

\bibitem{BJNRZ5} B. Barker, M. Johnson, P. Noble, M. Rodrigues, K. Zumbrun,
{\it Witham averaged equations and modulational stability of periodic solutions of hyperbolic-parabolic balance laws}, 
Proceedings, French GDR meeting on EDP, Port D'Albret, France.

\bibitem{BD}
N. Bottman and B. Deconinck.
KdV cnoidal waves are linearly stable. \emph{Discrete Contin. Dyn. Syst.} 25 (4): 1163-1180, 2009.

\bibitem{BHZ2} B. Barker, J. Humpherys, and K. Zumbrun, \textit{STABLAB: A MATLAB-based numerical library for Evans function computation}, Available at: \url{http://impact.byu.edu/stablab/}

\bibitem{BBCL} T. Bloom, L. P. Bos, J.-P. Calvi, and N. Levenberg,
{\it Polynomial interpolation and approximation in $\C^d$},
 \url{http://arxiv.org/pdf/1111.6418v1.pdf}.

\bibitem{BrJK}
J. C. Bronski, M. A. Johnson, and T. Kapitula. An Index Theorem for the Stability of Periodic Traveling Waves of KdV Type.
Proc. Roy. Soc. Edinburg Sect. A 141 (6): 1141-1173, 2011.

\bibitem{CD} H-C. Chang and E.A. Demekhin,
{\it  Complex wave dynamics on thin films,} (Elsevier, 2002).

\bibitem{CDK}
H.C. Chang, E.A. Demekhin, D.I. Kopelevich.
\newblock Laminarizing effects of dispersion in an active-dissipative nonlinear medium.
\newblock {\em Phys. D} 63: 299--320, 1993.

\bibitem{DY}
L. Demanet, L. Ying,
{\it On Chebyshev interpolation of analytic functions}, MIT technical report, March 2010.

\bibitem{EMR}
N.M. Ercolani, D.W. McLaughlin and H. Roitner,
{\it Attractors and transients for a perturbed periodic {K}d{V} equation:
  a nonlinear spectral analysis}, J. Nonlinear Sci. 3(4):477-539 (1993).

\bibitem{FST} U. Frisch, Z.S. She, and O. Thual,
{\it Viscoelastic behaviour of cellular solutions
to the Kuramoto--Sivashinsky model} J. Fluid Mech. 168 (198) 221--240.

\bibitem{G} R. Gardner, {\it On the structure of the spectra of
periodic traveling waves}, J. Math. Pures Appl. 72 (1993), 415-439.

\bibitem{Gu} R. Gunttner, 
{\it Evaluation of Lebesgue Constants},
SIAM Journal on Numerical Analysis, Vol 17, No 4 (1980), 512-520.

\bibitem{GAO} {\it PATRIOT MISSILE DEFENSE:
Software Problem Led to System Failure at Dhahran, Saudi Arabia}, Government Accountability Office report, available at \url{http://www.gao.gov/products/IMTEC-92-26}.

\bibitem{Gu} J. Guyker, 
{\it An inequality for Chebyshev connection coefficients}, 
J. Ineq. Pure and Applied Math 7(2) Art 67, 2006.

\bibitem{IB} {\it Ariane 5 Flight 501 Failure}, Report by the inquiry board, available at \url{http://esamultimedia.esa.int/docs/esa-x-1819eng.pdf}.

\bibitem{GR} I. S. Gradshteyn and I. M. Ryzhik,
{\it Table of integrals, series, and products},
Copyright 2007, Elseveir, Inc.

\bibitem{JNRZ} Mathew A. Johnson, Pascal Noble, L. Miguel Rodriguez, Kevin Zumbrun,
{\it Spectral stability of periodic wave trains of the Korteweg-de Vries/Kuramoto-Sivashinksy equation in the Korteweg-de Vries limit}, Transactions of the AMS, to appear.

\bibitem{JNRZ2} M. A. Johnson, P. Noble, L.M. Rodrigues, K. Zumbrun
{\it Behaviour of periodic solutions of viscous conservation laws under localized and nonlocalized perturbations.} to appear, Inventiones Math.

\bibitem{JZ1} M. Johnson and K. Zumbrun,
{\it Nonlinear stability and asymptotic behavior of periodic traveling waves
of multidimensional viscous conservation laws in dimensions one and two},
SIAM J. Appl. Dyn. Syst. 10 (2011), no. 1, 189--211. 

\bibitem{JZ2} M. Johnson and K. Zumbrun,
{\it Nonlinear stability of periodic traveling waves
of viscous conservation laws in the generic case},
Journal of Differential Equations, 249 no. 5 (2010), 1213-1240.

\bibitem{JZN} M. Johnson, K. Zumbrun, and P. Noble,
{\it Nonlinear stability of viscous roll waves,}
SIAM Journal on Mathematical Analysis, 43 (2011) no. 2, 577-611.

\bibitem{JNRZ1} M. A. Johnson, P. Noble, L.M. Rodrigues, K. Zumbrun
{\it Spectral stability of periodic wave trains of the Korteweg-de Vries/Kuramoto-Sivashinsky equation in the Korteweg-de Vries limit}, submitted (2012) arXiv:1202.6402.

\bibitem{K} Y. Kuramoto,
{\it Chemical oscillations, waves, and turbulence,}
Springer-Verlag, Berlin, 1984, 164 p.

\bibitem{M} J.C. Mason,
{\it Near-Best Multivariate Approximation by Fourier series, Chebyshev series and Chebyshev interpolation},
Journal of Approximation Theory, 28, 349-358 (1980).

\bibitem{KE} P. Kent and J. Elgin,
{\it Traveling-waves of the Kuramoto-Sivashinsky equation: period-multiplying bifurcations. }
Nonlinearity 5 (1992) no. 4, 899-919.


\bibitem{Kh} Boris N. Khoromskij,
University/ETH Z\"{u}rich, Pro∗Doc Program, WS 2010 lecture notes.

\bibitem{KSF}
E. A. Kuznetsov, M. D. Spector, and G. E. Fal'kovich. On the stability of nonlinear waves
in integrable models. \emph{Phys. D} 10: 379--386, 1984.

\bibitem{KT} Y. Kuramoto and T. Tsuzuki,
{\it On the formation of dissipative structures in reaction-diffusion systems,}
Progress of Theoretical Physics, 1975. 54:3.

\bibitem{N} Alexander A. Nepomnyashchy, 
{\it Kawahara equation},
Trans. Penn State Univ. 362 (1976) 114 [in Russian].

\bibitem{PSU} R. Pego, H. Schneider, and H. Uecker,
{\it Long-time persistence of Korteweg-de Vries
solitons as transient dynamics in a model of inclined film flow,}
Proc. Royal Soc. Edinburg 137A (2007), 133--146.

\bibitem{WR} 
{\it The best-known properties and formulas for Weierstrass functions and inverses}, 1998-2013 Wolfram Research, Inc., 
\begin{tiny}\url{http://functions.wolfram.com/EllipticFunctions/WeierstrassSigma4/introductions/Weierstrass/05/} \end{tiny}

\bibitem{R} S. M. Rump, 
{\it INTLAB-INTerval LABoratory},
in: Tibor Csendes (ed.), {\it Developments in Reliable Computing, Kluwer Academic Publishers}, Dordrecht, Netherlands, 1999, pp. 77-105.   

\bibitem{Rep} 
{\it ICERM Reproducibility in Computational and Experimental Mathematics: Readings and References},
Web page describing the proceedings of the ICERM conference on Reproducibility in Computational and Experimental Mathematics held in December 2013, available at \url{http://wiki.stodden.net/ICERM_Reproducibility_in_Computational_and_Experimental_Mathematics:_Readings_and_References}.

\bibitem{RW} S. C. Reddy and J. A. C. Weideman,
{\it The accuracy of the Chebyshev differencing method for analytic functions},
Siam Journal on Numerical Analysis 42(5) pg 2176-2184, 2005.

\bibitem{S} M. D. Spector,
{\it Stability of conoidal waves in media with positive and negative dispersion},
Sov. Phys. JETP 67(1) pg 104, 1988.

\bibitem{S1} G.I. Sivashinsky,
{\it  Nonlinear analysis of hydrodynamic instability in laminar flame. I. Derivation of basic equations,}
Acta Astron. 4 (1977) no. 11-12, 1177--1206.

\bibitem{S2} G.I. Sivashinsky,
{\it  Instabilities, Pattern Formation, and Turbulence in Flames},
Annual Review of Fluid Mechanics, January 1983. 15. Pp.179-199.

\bibitem{Sc1} G. Schneider, {\it Nonlinear diffusive stability
of spatially periodic solutions-- abstract theorem and higher space
dimensions},
Proceedings of the International Conference on Asymptotics
in Nonlinear Diffusive Systems (Sendai, 1997),  159--167,
Tohoku Math. Publ., 8, Tohoku Univ., Sendai, 1998.

\bibitem{Sc2} G. Schneider,
{\it Diffusive stability of spatial periodic solutions of the
Swift-Hohenberg equation,} (English. English summary)
Comm. Math. Phys. 178 (1996), no. 3, 679--702.

\bibitem{T} Eitan Tadmor,
{\it The Exponential Accuracy of Fourier and Chebyshev Differencing Methods},
SIAM Journal on Numerical Analysis, Volume 23, Issue 1 (Feb., 1986), 1-10.

\bibitem{W} H. A. Win,
{\it Model equation of surface waves of viscous fluid down an inclined plane},
J. Math. Kyoto Univ. 33 (1993) no. 3, 803-824.

\bibitem{YY} J. Yu, Y. Yang,
{\it Evolution of small periodic disturbances into roll waves in channel flow with internal dissipation},
Stud. Appl. Math. 111 (2003) no. 1, 1-27.

\bibitem{ZH}K. Zumbrun and P. Howard,
 {\it Pointwise semigroup methods and stability of viscous shock waves},
Indiana University Mathematics Journal, 47(3) pg 741-871, 1998.

\end{thebibliography}
\end{document}